\documentclass[oneside]{amsart}
\usepackage{amsmath}
\usepackage{amssymb}
\usepackage{amsthm}
\usepackage{float}
\usepackage{ulem}
\usepackage{enumerate}
\usepackage{xcolor}
\usepackage[utf8x]{inputenc}
\usepackage{comment}
\usepackage{framed}
\usepackage{caption}
\usepackage{todonotes}

\newcommand{\eps}{\varepsilon}

\DeclareMathAlphabet{\mathbbb}{U}{bbold}{m}{n}

\newtheorem{thm}{Theorem}[section]
\newtheorem{cor}[thm]{Corollary}

\newtheorem{lem}[thm]{Lemma}

\newtheorem{conj}[thm]{Conjecture}
\theoremstyle{definition}

\theoremstyle{remark}

\numberwithin{equation}{section}
\usepackage{graphicx}

\newcommand{\R}{{\mathbb R}}

\newcommand{\matrice}{\begin{pmatrix}}
\newcommand{\ok}{\end{pmatrix}}

\newcommand{\BBS}{\mathbb{S}}

\definecolor{ottanio}{RGB}{2, 142, 185}

\begin{document}
\title[Shape optimisation for Maxwell's eigenvalues]{On a shape optimisation problem for Maxwell's eigenvalues on cuboids}
\thanks{The first and third authors acknowledge support of the INdAM GNAMPA group. The second author acknowledges the support of the INdAM GNSAGA group.}
\author[Lamberti]{Pier Domenico Lamberti}
\address{Dipartimento di Tecnica e Gestione dei Sistemi Industriali (DTG), Universit\`a degli Studi di Padova, Stradella S. Nicola 3 - 36100 Vicenza, Italy, e-mail: {\sf pierdomenico.lamberti@unipd.it}.}
\author[Provenzano]{Luigi Provenzano}
\address{Dipartimento di Scienze di Base e Applicate per l'Ingegneria, Sapienza Universit\`a di Roma, Via Scarpa 12 - 00161 Roma, Italy, e-mail: {\sf luigi.provenzano@uniroma1.it}.}
\author[Sempio]{Rebecca Sempio}
\address{Dipartimento di Matematica ``Tullio Levi-Civita'', Universit\`a degli Studi di Padova, Via Trieste 63 - 35121 Padova, Italy, e-mail: {\sf rebecca.sempio@math.unipd.it}.}

\begin{abstract}
We consider an optimisation problem for the elementary symmetric functions of the first three Maxwell's eigenvalues on cuboids under volume and perimeter constraint, and we show that the cube is a local  minimiser. More precisely, it is the unique minimiser in an explicit  cone of cuboids. The result gives a model case for the local optimisation of Maxwell's eigenvalues. On the other hand we show that such local extremality phenomena cannot be expected outside rigid geometric classes.
\end{abstract}


\thanks{}

\maketitle

\section{Introduction and statement of the main results}
Let $\Omega$ be a bounded domain in $\mathbb R^3$. We consider the eigenvalue problem for the {\rm curl\,curl} operator:
\begin{equation}\label{curlcurl}
\begin{cases}
{\rm curl\,curl\,}u=\lambda u\,, & {\rm in\ }\Omega\,,\\
{\rm div\,}u=0\,, & {\rm in\ }\Omega\,,\\
u\times\nu=0\,, & {\rm on\ }\partial\Omega\,,
\end{cases}
\end{equation}
in the unknowns $u:\Omega\to\mathbb R^3$ and $\lambda\in\mathbb R$. Problem \eqref{curlcurl} arises from the second-order reformulation of the time-harmonic Maxwell's system. We refer e.g., to \cite{cessenat,hanson,monk,nedelec} for a mathematical treatment of the theory of electromagnetism. Other relevant references on Problem \eqref{curlcurl} are the papers of Costabel and Dauge \cite{costabel1,costabel2,Costabel_2019}. We also refer to the more recent \cite{cocomo,ferraresso_provenzano,lastr,LambertiZaccaron2020,pauly,yin} and references therein. 

If $\Omega$ has a Lipschitz boundary, then Problem \eqref{curlcurl} admits a sequence of non-negative eigenvalues of finite multiplicity
$$
0\leq\lambda_1(\Omega)\leq\lambda_2(\Omega)\leq\cdots\leq\lambda_j(\Omega)\leq\cdots\nearrow+\infty,
$$
which we call Maxwell's  eigenvalues. If $\partial\Omega$ is connected, then $\lambda_1(\Omega)>0$. Under this assumption, it is natural to consider maximisation and minimisation problems for $\lambda_1(\Omega)$ under volume or perimeter constraint. However, it is well-known that
$$
\inf_{|\Omega|=1}\lambda_1(\Omega)=0,\quad\quad\quad\inf_{|\partial\Omega|=1}\lambda_1(\Omega)=0
$$
and
$$
\sup_{|\Omega|=1}\lambda_1(\Omega)=+\infty,\quad\quad\quad\sup_{|\partial\Omega|=1}\lambda_1(\Omega)=+\infty.
$$
The failure of the Faber-Krahn and reverse Faber-Krahn inequalities under either a volume or a perimeter constraint is proved in \cite{lamberti_zaccaron_FK}.
In the case of volume constraint, the result was already contained in \cite{guerini,savo_convex}, formulated in terms of eigenvalues of the Hodge Laplacian acting on differential forms.

\smallskip

This fact is not surprising. As  observed e.g., in \cite{ferraresso_provenzano,lamberti_zaccaron_FK,LambertiZaccaron2020,savo_convex},  the behaviour of Maxwell's eigenvalues can be very {\it wild} when it comes to shape optimisation, and simple geometries already highlight this phenomenon.  In \cite{ferraresso_provenzano} it is proved that  Maxwell's eigenvalues of a thin tubular neighborhood of an embedded hypersurface of $\mathbb R^3$ converge to the eigenvalues of the Laplacian on the surface (with Dirichlet conditions if it has a boundary). This allows for a great flexibility of the eigenvalues, which, under the sole volume or perimeter constraint, can be made arbitrarily large or small, regardless of topological or geometrical restrictions (number of boundary components, inradius, diameter, etc.).

\smallskip

Typically, for operators that are invariant under isometries, it is expected that the optimum for the first eigenvalue is the ball, or, at least, that it is a critical domain. Differently from the case of the Dirichlet Laplacian, for Problem \eqref{curlcurl}, the first eigenvalue of the ball is not simple: it has multiplicity $3$. When treating multiple eigenvalues, it has been highlighted in many contexts (see e.g., \cite{bula2,bula1,lala04,lalumu,LambertiZaccaron2020}) that the most natural point of view is that of considering the {\it elementary symmetric functions of the eigenvalues} rather than the eigenvalues themselves. For Problem \eqref{curlcurl}, it is shown in \cite{LambertiZaccaron2020} that the symmetric functions of multiple eigenvalues are real analytic, and that balls are {\it critical points} under a volume or perimeter constraint. A natural question is then: 

\smallskip

{\it ``Is the ball a local extremum for the symmetric functions of the first three eigenvalues?''}

\smallskip

In this paper we address a simplified version of this question. Namely, we restrict to cuboids, and study the local extremality of the cube for the symmetric functions of the first three eigenvalues under volume and perimeter constraints.

\smallskip

More precisely, set
\begin{align*}
&F_1(a,b,c)=a+b+c\nonumber\\
&F_2(a,b,c)=ab+ac+bc\\
&F_3(a,b,c)=abc\nonumber
\end{align*}
where $a,b,c\in\mathbb R$. The functions $F_i$, $i=1,2,3$ are the elementary symmetric functions in three  variables $(a,b,c)$.

\smallskip

Let $\mathbb R^3_+:=\{(x,y,z)\in\mathbb R^3:x>0,y>0,z>0\}$. A {\it cuboid} $\Omega(\ell)$, $\ell=(\ell_1,\ell_2,\ell_3)\in\mathbb R^3_+$, is a cartesian product domain of $\mathbb R^3$ defined (up to isometries) as
$$
\Omega(\ell)=(0,\ell_1)\times(0,\ell_2)\times(0,\ell_3).
$$
Hence, to a point  $\ell\in\mathbb R^3_+$ we associate the cuboid $\Omega(\ell)$. 

We define the following region in $\mathbb R^3_+$: 
$$
\mathcal P=\{(\ell_1,\ell_2,\ell_3)\in\mathbb R^3_+:\max\{\ell_1,\ell_2,\ell_3\}<2\min\{\ell_1,\ell_2,\ell_3\}\}.
$$
 Note that the open half-line of points of coordinates $(L,L,L)$, $L>0$, lies in the interior of $\mathcal P$. This line corresponds to cubes.

For $L>0$, we also define  the following two regions in $\R^3_+$:
  \begin{align*}
     R_{V,L}:&= \{(\ell_1,\ell_2,\ell_3) \in\R^3_+:\,\ell_1\ell_2\ell_3=L^3\}\,,\\
     R_{P,L}:&=\{(\ell_1,\ell_2,\ell_3)\in \R^3_+:\,\ell_1\ell_2+\ell_1\ell_3+\ell_2\ell_3=3L^2\}.
 \end{align*}
We write $R_{\star,L}$ to denote either $R_{V,L}$ or $R_{P,L}$.

 Our main result is stated as follows.

\begin{thm}\label{main2}
Let $\ell=(\ell_1,\ell_2,\ell_3)\in\mathbb R_+^3$ and let $\Omega(\ell)=(0,\ell_1)\times (0,\ell_2)\times(0,\ell_3)$. Let $L>0$ be fixed. Assume that
$$
\ell\in\mathcal P\cap R_{\star,L}.
$$
Let $\ell_L=(L,L,L)$. Then
$$
F_i(\lambda_1(\Omega(\ell)),\lambda_2(\Omega(\ell)),\lambda_3(\Omega(\ell)))\geq F_i(\lambda_1(\Omega(\ell_L)),\lambda_2(\Omega(\ell_L)),\lambda_3(\Omega(\ell_L)))
$$
for $i=1,2,3$, and the equality holds if and only if $\ell=\ell_L$.
\end{thm}
Theorem~\ref{main2} states that among all cuboids with dimensions $\ell=(\ell_1,\ell_2,\ell_3)\in\mathcal P$ and volume $L^3$ or perimeter $6L^2$, the cube of side $L$ uniquely minimises all the elementary symmetric functions of the first three eigenvalues. We will show that, as long as $\ell$ belongs to $\mathcal P$, $F_i(\lambda_1(\Omega(\ell)),\lambda_2(\Omega(\ell)),\lambda_3(\Omega(\ell)))$ is a smooth function of $\ell$ (see Theorem~\ref{main1}). Since the line $\ell_1=\ell_2=\ell_3$ (corresponding to cubes) lies in the interior of $\mathcal P$, the elementary symmetric functions of the first three eigenvalues, constrained to the submanifold $R_{\star,L}$, are smooth near the minimum point $(L,L,L)$. See Figures~\ref{sym_vol} and \ref{sym_per}.  
\begin{figure}
    \centering
    \includegraphics[width=\textwidth]{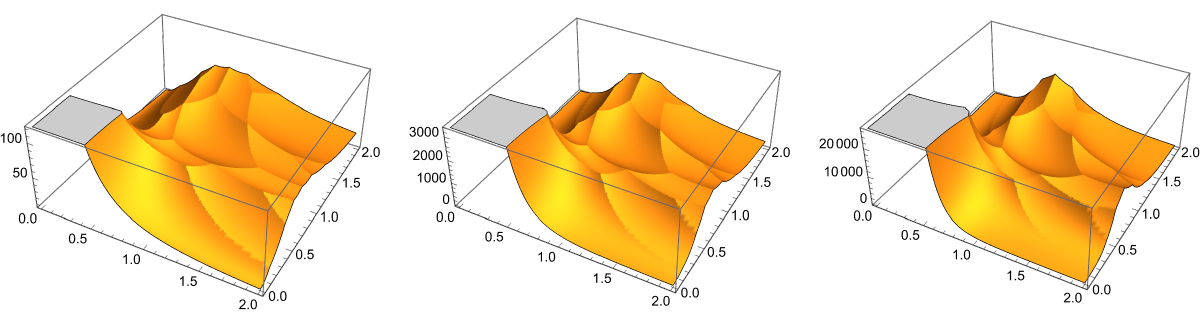}
    \caption{Plot of $F_i(\lambda_1(\Omega(\ell)),\lambda_2(\Omega(\ell)),\lambda_3(\Omega(\ell)))$, $i=1,2,3$, as functions of $(\ell_1,\ell_2)$. The volume constraint is encoded by setting $\ell=(\ell_1,\ell_2,\frac{1}{\ell_1\ell_2})$.}
    \label{sym_vol}
\end{figure}

\begin{figure}
    \centering
    \includegraphics[width=\textwidth]{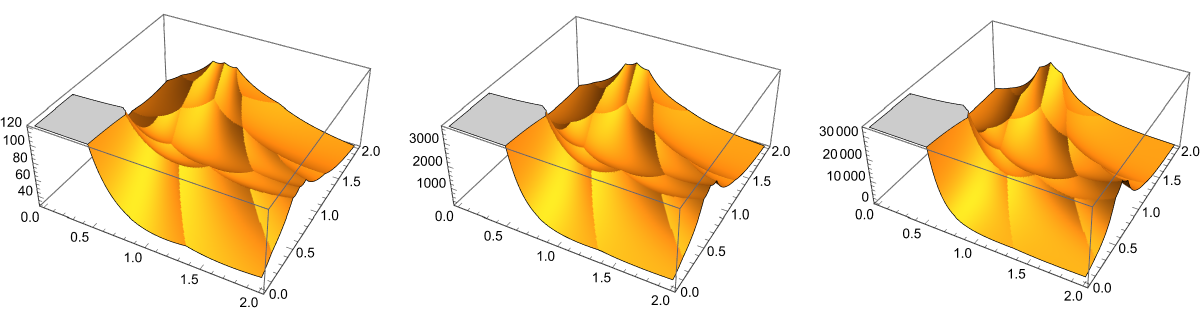}
    \caption{Plot of $F_i(\lambda_1(\Omega(\ell)),\lambda_2(\Omega(\ell)),\lambda_3(\Omega(\ell)))$, $i=1,2,3$, as functions of $(\ell_1,\ell_2)$. The perimeter constraint is encoded by setting $\ell=(\ell_1,\ell_2,\frac {3-\ell_1\ell_2}{\ell_1+\ell_2})$.}
    \label{sym_per}
\end{figure}

\smallskip

This model situation leads naturally to the following question:

\smallskip

{\it Is the ball a local minimiser under volume or perimeter constraint for the symmetric functions of the first three eigenvalues among all domains homeomorphic to a ball?}

\smallskip

Clearly one cannot drop the topological hypothesis. In fact, if a domain has $b+1$ connected components of the boundary, then the spectrum contains the eigenvalue $0$ with multiplicity $b$. In Section~\ref{app:B} we prove that the answer to the question above is negative. More precisely, we prove the following
\begin{thm}\label{prop0}
Let $B$ denote the unit ball in $\mathbb R^3$. For any $N\in\mathbb N$, $N\geq 1$ and any $\eps>0$ there exists a domain $\Omega_{N,\eps}$ homeomorphic to $B$ with either $|\Omega_{N,\eps}|=|B|$ or $|\partial\Omega_{N,\eps}|=|\partial B|$ such that
$$
\lambda_N(\Omega_{N,\eps})\leq\eps\,,\ \ \ d_H(B,\Omega_{N,\eps})\leq \eps,
$$
where $d_H$ is the Hausdorff distance.
\end{thm}

One can replace $B$ in Theorem~\ref{prop0} by any other smooth domain $\Omega$, see Section~\ref{app:B}. This means that there are no local minimisers for the first eigenvalue, nor for the symmetric functions of any multiple eigenvalue, under volume or perimeter constraint. This example motivates our restriction to cuboids: outside a very rigid class of geometries, everything can happen at the level of shape optimisation, and local minimality of the ball fails. Hence, the restriction to very special geometric classes is not merely a technical simplification. Cuboids represent the most basic {\it convex} model in shape optimisation problems. Theorem~\ref{main2} and Theorem~\ref{prop0} motivate the following:

\begin{conj}\label{conj2} The ball is a local minimiser under volume or perimeter constraint for the symmetric functions of the first three eigenvalues among  {\rm convex} domains.
\end{conj}

We remark that even within the class of cuboids, global extremisers under volume constraint do not exist: all Maxwell's eigenvalues of $(0,\eps)\times(0,\eps)\times(0,\eps^{-2})$ go to $+\infty$ as $\eps\to 0$, while all Maxwell's eigenvalues of $(0,\eps^2)\times(0,\eps^{-1})\times(0,\eps^{-1})$ go to $0$ as $\eps\to 0$. In general it is known that `cigar-like' convex domains have large Maxwell's eigenvalues, while `flying saucer-like' convex domains have small Maxwell's eigenvalues, see \cite{savo_convex}. In the case of perimeter constraint, all Maxwell's eigenvalues can be made arbitrarily large as well by taking thin and long cuboids, while there is a uniform lower bound for the eigenvalues of convex domains in terms of the perimeter, which is a consequence of \cite{savo_convex}. However, it is very likely that among convex domains there are no global minimisers for the eigenvalues or their symmetric functions, and that minimising sequences degenerate in the limit (see Section~\ref{app:A} for more details in the case of cuboids). 

What is known for convex sets are {\it upper and lower bounds} for the first eigenvalue, see \cite{savo_convex}: it can be estimated from above and below in terms of  $\frac{1}{a_2^2}$, where $a_2>0$ is the intermediate semiaxis of the John ellipsoid of the set. This is the best meaningful geometric bound currently available.

Our study of cuboids is exactly in this spirit: we consider a specific geometric setting in which certain spectral optimisation questions are still meaningful and can be analysed explicitly.

\subsection{The first three eigenvalues of cuboids} The fact that the elementary symmetric functions of the first three eigenvalues of cuboids are smooth near the cube and that the cube is a local minimum under volume and perimeter constraint should be compared with the behaviour of each single eigenvalue $\lambda_i(\Omega(\ell))$, $i=1,2,3$.  In Section~\ref{app:A} we prove the following result:

\begin{thm}
\label{AutovSingolarmente}

 The functions $\lambda_k(\Omega(\cdot)):R_{\star,L}\to\mathbb R$, $k=1,2,3$  
    are Lipschitz continuous in a neighborhood of $\ell_L=(L,L,L)$, and are not differentiable at $\ell_L$. Moreover:
    \begin{enumerate}[a)]
    \item there exists a neighborhood $\mathcal U$ of $\ell_L$ in $R_{\star,L}$ such that 
    \[
    \lambda_1(\Omega(\ell_L))\ge \lambda_1(\Omega(\ell)) \quad \text{and} \quad \lambda_3(\Omega(\ell_L))\le \lambda_3(\Omega(\ell)) 
    \]
    for every $\ell \in \mathcal U$. Equality holds for $\ell\in\mathcal U$ if and only if $\ell=\ell_L$;
    \item there exist $\ell',\ell'' \in \mathcal U$ such that
    \[
    \lambda_2(\Omega(\ell'))< \lambda_2(\Omega(\ell_L)) < \lambda_2(\Omega(\ell'')).
    \]
    \end{enumerate}
\end{thm}
Summarising, Theorem~\ref{AutovSingolarmente} states that, among cuboids of fixed volume or perimeter:
    \begin{itemize}
    \item the cube is a local maximiser for $\lambda_1(\Omega(\ell))$;
    \item  the cube is not a local maximiser nor a local minimiser for $\lambda_2(\Omega(\ell))$;
    \item the cube is a local minimiser for $\lambda_3(\Omega(\ell))$;
    \item each single eigenvalue is not smooth at the cube, but their combination in some elementary symmetric function is smooth, and for such combinations, the cube is a minimiser. 
    \end{itemize}

At this point one is led to formulate the following naive question:

\smallskip

{\it Is the ball a local maximiser under volume or perimeter constraint for the first eigenvalue among all domains homeomorphic to a ball?}
\smallskip

In Section~\ref{app:B} we discuss why this question should have a negative answer. In particular, we discuss how, for any $\eps>0$, one can construct a domain $\Omega_\eps$ diffeomorphic to $B$  with either $|\Omega_\eps|=|B|$ or $|\partial\Omega_\eps|=|\partial B|$ such that $\lambda_1(\Omega_\eps)>\lambda_1(B)\,,\ \ \ d_H(B,\Omega_\eps)\leq \eps$.  Such an example would show that there are no local maximisers for the first Maxwell's eigenvalue under volume or perimeter constraint. However, we don't give here a rigorous proof of such construction since it would fall outside the scopes of the present paper. At any rate, it is natural to state the more meaningful:

\begin{conj}\label{conj4}The ball is a local maximiser under volume or perimeter constraint for the first eigenvalue among  {\rm convex} domains.
\end{conj}

\subsection{Organisation of the paper}
The paper is organised as follows. In Section~\ref{sec2} we present the proof of Theorem~\ref{main2}, which is based on two results, Theorems~\ref{main0} and \ref{main1}, which are proved in Sections~\ref{sec:main0} and \ref{sec:main1}, respectively. In Section~\ref{app:A} we discuss the behaviour of the first three eigenvalues on cuboids under volume and perimeter constraints, showing that the cube is a local maximiser for the first eigenvalue and a local minimiser for the third eigenvalue. In Section~\ref{app:B} we prove that the ball is not a local minimiser for the first eigenvalue, and for the symmetric functions of multiple eigenvalues, showing that local perturbation can produce an arbitrary number of small eigenvalues. We also discuss the case of the local maximization.

\section{Proof of Theorem~\ref{main2}}\label{sec2}

In order to prove Theorem~\ref{main2}, we first study a constrained optimisation problem for the symmetric functions $F_1,F_2,F_3$ of three numbers depending on $\ell=(\ell_1,\ell_2,\ell_3)\in\mathbb R^3_+$, which are eigenvalues of \eqref{curlcurl} on cuboids, but not necessarily the first three. Precisely, set
\begin{align}\label{symmetric_F}
&a(\ell):=\pi^2\left(\frac{1}{\ell_2^2}+\frac{1}{\ell_3^2}\right)\nonumber\\
&b(\ell):=\pi^2\left(\frac{1}{\ell_1^2}+\frac{1}{\ell_3^2}\right)\\
&c(\ell):=\pi^2\left(\frac{1}{\ell_1^2}+\frac{1}{\ell_2^2}\right).\nonumber
\end{align}
In Section~\ref{sec:main0} we prove the following.

\begin{thm}\label{main0}
Let $\ell=(\ell_1,\ell_2,\ell_3)\in\mathbb R_+^3$. Assume that $\ell$ satisfies one of the following two conditions:
\begin{enumerate}[a)]
\item $\ell_1\ell_2\ell_3=1$.
\item$\ell_1\ell_2+\ell_1\ell_3+\ell_2\ell_3=3$.
\end{enumerate}
Let $\ell_0=(1,1,1)$. Then
\begin{equation}\label{ineq_area}
F_i(a(\ell),b(\ell),c(\ell))\geq F_i(a(\ell_0),b(\ell_0),c(\ell_0))
\end{equation}
for $i=1,2,3$, and the equality holds if and only if $\ell=\ell_0$.
\end{thm}

Theorem~\ref{main0} establishes that the elementary symmetric functions of the three numbers $a(\ell),b(\ell),c(\ell)$ have a unique global minimum under either the constraint a) or b), and it is achieved if and only if $\ell=\ell_0=(1,1,1)$.

\smallskip

Let $\ell\in\mathbb R^3_+$ and let $\Omega(\ell)$ be the corresponding cuboid. The volume of $\Omega(\ell)$ is  given by
$$
|\Omega(\ell)|=\ell_1\ell_2\ell_3
$$
and its perimeter (the total surface area) is given by
$$
2(\ell_1\ell_2+\ell_1\ell_3+\ell_2\ell_3).
$$
When $\ell=\ell_0=(1,1,1)$ we have the unit cube, which has volume $1$ and perimeter $6$.

\smallskip

Next, we relate the numbers $a(\ell),b(\ell),c(\ell)$ with the eigenvalues of \eqref{curlcurl} on $\Omega(\ell)$.

The spectrum of a cuboid is well-known, see \cite{Costabel_2019}. The eigenvalues are given by the following families
\begin{enumerate}[i)]
    \item $\pi^2\left(\frac{k_1^2}{\ell_1^2}+\frac{k_2^2}{\ell_2^2}+\frac{k_3^2}{\ell_3^2}\right)$  counted once, with $k_i\in\mathbb N$, and exactly one $k_i=0$;
    \item $\pi^2\left(\frac{k_1^2}{\ell_1^2}+\frac{k_2^2}{\ell_2^2}+\frac{k_3^2}{\ell_3^2}\right)$ counted twice, with $k_i\in\mathbb N$, $k_i\geq 1$.
\end{enumerate}

Here by $\mathbb N$ we denote the set of natural numbers, including $0$. If $\ell_0=(1,1,1)$, then we see that
$$
a(\ell_0)=b(\ell_0)=c(\ell_0)=\lambda_1(\Omega(\ell_0))=\lambda_2(\Omega(\ell_0))=\lambda_3(\Omega(\ell_0)).
$$
In other words, $a(\ell_0),b(\ell_0),c(\ell_0)$ are exactly the first three eigenvalues of \eqref{curlcurl} for the unit cube $\Omega(\ell_0)$.  
 
 The next theorem, which we prove in Section~\ref{sec:main1}, characterizes the cuboids $\Omega(\ell)$ for which the numbers $a(\ell),b(\ell),c(\ell)$ are  the first three eigenvalues: they correspond to $\ell\in\overline{\mathcal P}$.

\begin{thm}\label{main1}
We have
$$
\{\lambda_1(\Omega(\ell)),\lambda_2(\Omega(\ell)),\lambda_3(\Omega(\ell))\}=\{a(\ell),b(\ell),c(\ell)\}
$$
as sets of numbers if and only if $\ell=(\ell_1,\ell_2,\ell_3)\in\overline{\mathcal P}$. If moreover $0<\ell_1\leq\ell_2\leq\ell_3$, then $\lambda_1(\Omega(\ell))=a(\ell)$, $\lambda_2(\Omega(\ell))=b(\ell)$, $\lambda_3(\Omega(\ell))=c(\ell)$.
\end{thm}

Theorem~\ref{main2} follows from Theorems~\ref{main0} and \ref{main1}, and from the scaling of the eigenvalues: for any $L>0$
$$
\lambda_j(L\Omega)=L^{-2}\lambda_j(\Omega).$$
\qed

\section{Proof of Theorem~\ref{main0}}\label{sec:main0}

The idea of the proof is the same for each of the functions $F_i$, $i=1,2,3$ and for both the volume and perimeter constraints. First, we prove that the only critical point for $F_i$ under the volume/perimeter constraint inside a suitable compact set $K\subset \R_+^3$ is given by $\ell=\ell_0$. Then we prove that outside $K$ the values of $F_i$ under the constraint are strictly larger than those achieved by $\ell_0$. This argument proves that the minimum of each of the functions $F_i$ is achieved at $\ell_0$ under volume or perimeter constraint.

\subsection{Volume constraint}

Here we consider $\ell\in \R^3_+$ satisfying the volume constraint
\[
V(\ell):= \ell_1\ell_2\ell_3 -1 = 0.
\]
 If we fix the set $K_M:=\{(\ell_1,\ell_2,\ell_3) \in \R_+^3: \ell_i\le M \;\forall i=1,2,3\}$ for an appropriate $M>0$ (to be chosen later), then $K_M\cap\{V(\ell)=0\}$ is compact. By the Weierstrass Theorem each of the continuous functions $F_i$ attains a minimum in $K_M\cap \{V(\ell)=0\}$. We prove that such minimum in $K_M\cap \{V(\ell)=0\}$ is attained at $\ell_0$, and that $\ell_0$ is also a global minimiser for each of the functions $F_i$, $i=1,2,3$.

\subsubsection{First symmetric function}
We consider $\ell\in {\rm Int}K_M$ with $M> \frac 32$, where by ${\rm Int}$ we denote the interior of a set. We recall that the first symmetric function is given by
\[
F_1(\ell) = a(\ell) + b(\ell) + c(\ell) = 2\pi^2 \Big(\frac{1}{\ell_1^2} +  \frac{ 1}{\ell_2^2} +\frac{1}{\ell_3^2}\Big).
\]
Using the method of Lagrange multipliers, if $\ell$ is a constrained critical point, then there exists $\eta\in \R$ such that
\[
\begin{cases}
    &\nabla F_1(\ell) = \eta\, \nabla V(\ell),\\
    & \ell_1 \ell_2 \ell_3  = 1.
\end{cases}
\]
By explicit computations we have
\[
\begin{cases}
    &-4 \pi^2  = \eta\, \ell_1^3 \ell_2 \ell_3= \eta  \ell_1^2,\\
    &-4 \pi^2  = \eta\, \ell_1 \ell_2^3 \ell_3= \eta \ell_2^2,\\
    &-4 \pi^2  = \eta\, \ell_1 \ell_2 \ell_3^3= \eta \ell_3^2,\\
    & \ell_1 \ell_2 \ell_3 = 1.
\end{cases}
\]
Clearly the only possibility is that $\ell_1 =\ell_2 = \ell_3$, meaning that $\ell_i=1$ for $i=1,2,3$, and $\ell=\ell_0$ is the only critical point in ${\rm Int} K_M$.

Now we consider a point $\ell \in \R^3_+\setminus {\rm Int} K_M$. Hence, at least one of its coordinates must be larger than $M$.  Without loss of generality we may assume that $\ell_1\le\ell_2\le\ell_3$, and then $\ell_3\ge M$.
By using the volume constraint, we can write $\ell_1 = \frac{1}{\ell_2\ell_3}$ and
\[
F_1(\ell) = 2\pi^2 \Big(\ell_2^2\ell_3^2+\frac{1}{\ell_2^2}+\frac{1}{\ell_3^2}\Big) >
2\pi^2 \Big( \ell_2^2\ell_3^2+\frac{1}{\ell_2^2}\Big).
\]
By using the fact that for every $a,b\ge 0$, $a^2 + b^2 \ge 2 ab$ and that $M> \frac32$, we have
\[
F_1(\ell) > 4 \pi^2\ell_3  \ge 4\pi^2M >6\pi^2,
\]
which is the value of $F_1$ for $\ell_0$.

\subsubsection{Second symmetric function}

We consider $\ell\in {\rm Int} K_M$ with $M>2\sqrt{3}$. We recall that the second symmetric function is given by
\begin{equation}
    \label{formula F_2 vol constraint}
\begin{split}
    F_2(\ell) &= a(\ell)b(\ell)+b(\ell)c(\ell)+a(\ell)c(\ell) \\&=
        \pi^4 \bigg(\frac{1}{\ell_2^2} +  \frac{1}{\ell_3^2} \bigg) \bigg(\frac{1}{\ell_1^2} + \frac{1}{\ell_3^2} \bigg)+\pi^4
         \bigg(\frac{1}{\ell_1^2} +  \frac{1}{\ell_3^2} \bigg) \bigg(\frac{1}{\ell_1^2} + \frac{1}{\ell_2^2} \bigg)\\&\qquad \qquad+
        \pi^4\bigg(\frac{1}{\ell_2^2} +  \frac{1}{\ell_3^2} \bigg) \bigg(\frac{1}{\ell_1^2} + \frac{1}{\ell_2^2} \bigg).
\end{split}
\end{equation}
Again, if $\ell$ is a critical point under volume constraint, there exists $\eta\in \R$ such that
\[
\begin{cases}
    &\nabla F_2(\ell) = \eta\, \nabla V(\ell),\\
    & \ell_1 \ell_2 \ell_3  = 1.
\end{cases}
\]
Now, $\partial_{\ell_1}F_2(\ell)-\eta\partial_{\ell_1}V(\ell)$ reads
\begin{equation}\label{first_c}
-  \frac{2\pi^2}{\ell_1^3} \bigg(2 \Big(\frac{\pi}{\ell_1} \Big)^2 + 3\Big( \frac{\pi}{\ell_3} \Big)^2 +3 \Big( \frac{\pi}{\ell_2} \Big)^2 \bigg) = \eta \ell_2 \ell_3.
\end{equation}
Using the following notation
\begin{equation}
\label{def ABC}
    A:= \Big( \frac{\pi}{\ell_1} \Big)^2 \quad B:= \Big( \frac{\pi}{\ell_2} \Big)^2 \quad C:= \Big( \frac{\pi}{\ell_3} \Big)^2
\end{equation}
and the fact that $\ell_1 \ell_2 \ell_3 = 1$, \eqref{first_c} can be written as
\[
- 2 A (2A+3B+3C) = \eta \ell_1 \ell_2 \ell_3 = \eta.
\]
Repeating the same computation for the other two components of $\nabla F_2(\ell)=\eta\nabla V(\ell)$, we get
\begin{equation}
    \label{Second,volume,eq1}
\begin{cases}
    & - 2 A (2A+3B+3C) = \eta, \\
    &- 2 B (2B+3A+3C) =  \eta ,\\
    &- 2 C (2C+3A+3B) = \eta .
\end{cases}
\end{equation}
By matching the first two equations we get
\[
2A^2 + 3AB + 3AC = 2B^2 + 3AB + 3BC,
\]
yielding
\[
(A-B) (2 (A+B) + 3C) = 0.
\]
Since $A,B,C>0$ it follows that $A=B$. By matching all the equations in \eqref{Second,volume,eq1}, we get that $A=B=C$, implying that $\ell_i = 1$ for $i=1,2,3$.

Now we estimate the value of $F_2$ at a point $\ell \in \R^3_+\setminus {\rm Int} K_M$.  By neglecting the last two terms in the right hand side of \eqref{formula F_2 vol constraint} we get
\[
\begin{split}
    F_2(\ell)
    \ge \pi^4 
    \frac{\ell_3^4+\ell_1^2\ell_2^2+\ell_1^2\ell_3^2+\ell_2^2\ell_3^2}{\ell_3^2(\ell_1\ell_2\ell_3)^2} \ge \pi^4\, \ell_3^2 .
\end{split}
\]
 We assume without loss of generality that $\ell_1\le\ell_2\le\ell_3$. Since \(\ell\not\in {\rm Int}K_M\), we then have \(\ell_3 \ge M\). Recalling that \(M>2\sqrt{3}\), we obtain
\[
F_2(\ell)\ge \pi^4 M^2 > 12 \pi^4,
\]
which is the value of \(F_2(\ell)\) for $\ell=\ell_0$.

\subsubsection{Third symmetric function}
 We recall that the third symmetric function is given by
\[
F_3(\ell) =a(\ell)b(\ell)c(\ell)=\pi^6 \left(\frac{1}{\ell_2^2}+\frac{1}{\ell_3^2}\right) \left(\frac{1}{\ell_1^2}+\frac{1}{\ell_3^2}\right)
\left(\frac{1}{\ell_1^2}+\frac{1}{\ell_2^2}\right).
\]
We consider $\ell\in{\rm Int}K_M$ with $M>2$. If $\ell$ is a constrained critical point, then there exists $\eta\in \R$ such that
\[
\begin{cases}
    &\nabla F_3(\ell) = \eta\, \nabla V(\ell),\\
    & \ell_1 \ell_2 \ell_3  = 1.
\end{cases}
\]
Now, $\partial_{\ell_1}F_3(\ell)-\eta\partial_{\ell_1}V(\ell)$  reads
\[
 -\frac{ 2\pi^2}{\ell_1^3} \bigg( \Big( \frac{\pi}{\ell_2} \Big)^2 +  \Big( \frac{ \pi}{\ell_3} \Big)^2 \bigg) 
 \bigg(2 \Big( \frac{\pi}{\ell_1} \Big)^2 + \Big( \frac{ \pi}{\ell_2} \Big)^2 + \Big( \frac{ \pi}{\ell_3} \Big)^2 \bigg)  = \eta \ell_2 \ell_3.
\]
Defining $A,B,C$ as in \eqref{def ABC} and using the volume constraint, the equation is rewritten as
\[
-2 A (B+C)(2A+B+C) = \eta \ell_1 \ell_2 \ell_3 = \eta.
\]
Repeating the argument for the other components of $\nabla F_3(\ell)=\eta\nabla V(\ell)$ we find
\[
\begin{cases}
    &-2 A (B+C)(2A+B+C)=\eta ;\\
    & -2 B (A+C)(2B+A+C)= \eta ;\\
    & -2 C (A+B)(2C+A+B) = \eta.
\end{cases}
\]
By matching the first two equations we have
\[
2A^2B + 2A^2C + AB^2 + 2 ABC  + AC^2 = 2AB^2 + 2B^2C + A^2B + 2 ABC  + BC^2,
\]
so that
\[
(A-B) [AB+2C (A+B)+C^2] = 0.
\]
Since $A,B,C>0$, we get $A = B$. By matching also the other equations we obtain that $A=B=C$, implying that $\ell_i = 1$ for $i=1,2,3$.

Now we estimate $F_3(\ell)$ when $\ell\in \R^3_+\setminus {\rm Int}K_M$. Without loss of generality, we may assume $\ell_1\le\ell_2\le\ell_3$, and thus $\ell_3\ge M$. By using the volume constraint $\ell_1=\frac{1}{\ell_2\ell_3}$, the fact that $\ell_1\leq\ell_2$ and
\begin{multline*}
F_3(\ell) = \pi^6 \frac{(\ell_1^2+\ell_2^2)(\ell_1^2+\ell_3^2)(\ell_2^2+\ell_3^2)}{(\ell_1\ell_2\ell_3)^4}\\=\pi^6 \bigg(\frac{1}{\ell_2^2\ell_3^2}+\ell_2^2\bigg)\bigg(\frac{1}{\ell_2^2\ell_3^2}+\ell_3^2\bigg)(\ell_2^2+\ell_3^2) \ge \pi^6\ell_2^2\ell_3^2(\ell_2^2+\ell_3^2),
\end{multline*}
we find
$$
F_3(\ell)\ge\pi^6\ell_3^3>8\pi^6,
$$
 which is the value of $F_3$ in $\ell_0$, since $M>2$ .

\subsection{Perimeter constraint}

Here we consider $\ell\in \R^3_+$ satisfying the perimeter constraint
\[
P(\ell):= \ell_1\ell_2+\ell_1\ell_3+\ell_2\ell_3 -3 = 0.
\]
Let
\[
K_M:=\{(\ell_1,\ell_2,\ell_3) \in \R_+^3:  \frac{1}{M}\le \ell_i\le M \;\forall i=1,2,3\}
\]
where $M>0$ will be fixed later. By the Weierstrass Theorem, each of the continuous functions $F_i$ attains a minimum in $K_M\cap \{P(\ell)=0\}$. We prove that such minimum in $K_M\cap \{P(\ell)=0\}$ is attained at $\ell=\ell_0=(1,1,1)$, and that $\ell_0$ is also a global minimiser for each of the $F_i$, $i=1,2,3$.

\subsubsection{First symmetric function}
We consider $\ell\in {\rm Int} K_M$ with $M> 3 \sqrt{3/2} $.
If $\ell$ is a constrained critical point, then there exists $\eta\in \R$ such that
\[
\begin{cases}
    &\nabla F_1(\ell) = \eta\, \nabla P(\ell),\\
    &\ell_1\ell_2+\ell_1\ell_3+\ell_2\ell_3 =3 .
\end{cases}
\]
Taking the first component in the equation $\nabla F_1(\ell)=\eta\nabla P(\ell)$, we get
\[
-4\frac{\pi^2}{\ell_1^3} =  \eta (\ell_2+\ell_3),
\]
and using the perimeter constraint we have
\[
- 4 \pi^2 = \eta (\ell_1 \ell_2 + \ell_1\ell_3 )\ell_1^2 =  \eta (3-\ell_2\ell_3) \ell_1^2.
\]
Doing the same for the other components of $\nabla F_1(\ell)=\eta\nabla P(\ell)$ we get
\begin{equation}\label{sysF1P}
\begin{cases}
    & \frac{- 4 \pi^2}{\eta} =  (3-\ell_2\ell_3) \ell_1^2;\\
    & \frac{- 4 \pi^2}{\eta} = (3-\ell_1\ell_3) \ell_2^2;\\
    & \frac{- 4 \pi^2}{\eta} = (3-\ell_1\ell_2) \ell_3^2;\\
    & \ell_1 \ell_2 + \ell_2 \ell_3 +\ell_1 \ell_3 = 3.
\end{cases}
\end{equation}
 We assume without loss of generality that $\ell_1\le\ell_2\le\ell_3$. Taking pairwise differences of the equations in system \eqref{sysF1P}, we obtain
 \begin{equation}\label{system_per}
\begin{cases}
    & \big(3(\ell_1+\ell_2) - \ell_1\ell_2 \ell_3 \big)(\ell_1-\ell_2) = 0;\\
    & \big(3(\ell_1+\ell_3) - \ell_1\ell_2 \ell_3 \big)(\ell_1-\ell_3) = 0;\\
    & \big(3(\ell_2+\ell_3) - \ell_1\ell_2 \ell_3 \big)(\ell_2-\ell_3) = 0.
\end{cases}
\end{equation}
If $\ell_1=\ell_2=\ell_3$ then we are done, since the constraint imposes $\ell_i=1$. Otherwise, since we assumed $\ell_1\le\ell_2\le\ell_3$, we surely have that $\ell_1\neq \ell_3$, and at least one of the two conditions $\ell_1\neq \ell_2$ or $\ell_2\neq \ell_3$ is satisfied. Assume for example that $\ell_1\neq\ell_2$ (the other case is analogous). Then, system \eqref{system_per} yields
\[
3(\ell_1+\ell_2) = \ell_1\ell_2 \ell_3 \quad \text{and} \quad 3(\ell_1+\ell_3) = \ell_1\ell_2 \ell_3,
\]
that implies $\ell_2=\ell_3$. However, the first equation in system \eqref{system_per} and the perimeter constraint give us the system
$$
\begin{cases}
3(\ell_1+\ell_2) = \ell_1\ell_2^2 ,\\
2\ell_1\ell_2+\ell_2^2=3,
\end{cases}
$$
that does not admit any real solution. Thus, the only possibility is that $\ell_1=\ell_2=\ell_3=1$.

Now,  take a point $\ell \in \R^3_+\setminus {\rm Int}K_M$. We may assume without loss of generality that $\ell_1\le\ell_2\le\ell_3$. Then either one of the three edges is larger than $M$ or one of the edges is smaller than $1/M$ (in our notation $\ell_3\ge M$ or $\ell_1 \le \frac{1}{M}$).
Let us consider the first case, namely the one where $l_3\ge M$. From the perimeter constraint we deduce that $\ell_i\ell_j <3$ for $i\neq j$. In particular, this implies that $\ell_1<3/\ell_3$ and $\ell_2<3/\ell_3$, and then
\[
F_1(\ell)
>
2\pi^2 \Big( 2 \frac{\ell_3^2}{9}+\frac{1}{\ell_3^2} \Big) 
>
\frac{4\pi^2}{9}\ell_3^2. 
\]
Since $\ell_3\ge M >3 \sqrt{3/2}$, we obtain
\[
F_1(\ell) >6\pi^2,
\]
which is the value of $F_1$ for $\ell_0$.
Now we study the case where $\ell_1 \le \frac{1}{M}$. Since $\ell_2^2,\ell_3^2\ge 0$ and $M \ge 3\sqrt{3/2}$ we have
\[
F_1(\ell) \ge \frac{2 \pi^2 }{\ell_1^2}\ge
2 \pi^2 M^2 \ge 27 \pi^2,
\]
which is larger than the value of $F_1(\ell_0)=6\pi^2$.

\subsubsection{Second symmetric function}

We consider $\ell\in {\rm Int} K_M$ with $M>3 \sqrt[4]{12} $.
If $\ell$ is a constrained critical point, then there exists $\eta\in \R$ such that
\[
\begin{cases}
    &\nabla F_2(\ell) = \eta\, \nabla P(\ell),\\
    &\ell_1\ell_2+\ell_1\ell_3+\ell_2\ell_3 =3 .
\end{cases}
\]
Using the perimeter constraint, the first component of $\nabla F_2(\ell)=\eta\nabla P(\ell)$ reads 
\[
-  \frac{2\pi^2}{\ell_1^2} \bigg(2 \Big(\frac{\pi}{\ell_1} \Big)^2 + 3\Big( \frac{\pi}{\ell_3} \Big)^2 +3 \Big( \frac{\pi}{\ell_2} \Big)^2 \bigg) = \eta (\ell_2+ \ell_3)\ell_1 =  \eta (3-\ell_2\ell_3),
\]
from which we deduce that $\eta<0$. We note that this equation is equivalent to
\[
-2\frac{\pi^4}{\ell_1^2(\ell_1\ell_2\ell_3)^2} (2\ell_2^2\ell_3^2+3\ell_1^2\ell_3^2+3\ell_1^2\ell_2^2) +\eta \ell_2\ell_3= 3\eta .
\]
Set
\begin{equation}
\label{def abc}
    a := \ell_2\ell_3, \quad b:= \ell_1\ell_3, \quad c:=\ell_1\ell_2,
\end{equation}
so that the equation is rewritten as
\[
- \frac{2\pi^4}{(bc)^2} \big( 2a^2+3b^2+3c^2 \big) + \eta a  = 3 \eta .
\]
Repeating the same procedure for the other components of the equation $\nabla F_2(\ell)=\eta\nabla P(\ell)$, we get
\[
\begin{cases}
    &- \frac{2\pi^4}{(bc)^2} \big( 2a^2+3b^2+3c^2 \big) + \eta \,a = 3\eta ,\\
    &-\frac{2\pi^4}{(ac)^2} \big( 2b^2+3a^2+3c^2 \big) + \eta \,b = 3\eta ;\\
    &- \frac{2\pi^4}{(ab)^2} \big( 2c^2+3a^2+3b^2 \big) + \eta \,c = 3\eta .\\
\end{cases}
\]
By subtracting the second equation to the first one we have
\[
- 2\pi^4 \big( 2\,(a^4-b^4) +3(a^2-b^2)\, c^2 \big) + \eta \,(a-b) (abc)^2 = 0.
\]
If we assume by contradiction that $a\neq b$ we can divide the above equality by $a-b$ and we have
\[
-2\pi^4 \big( 2\,(a^2+b^2)(a+b) +3(a+b)\, c^2 \big) + \eta \, (abc)^2 =0,
\]
deducing that $\eta>0$, which is a contradiction. Consequently, $a=b$ that implies that $\ell_1= \ell_2$. Using the other equations we obtain $\ell_i=1$ for $i=1,2,3$. 

Now, if we take a point $\ell \in \R^3_+\setminus {\rm Int}K_M$ and we assume  that $\ell_1\le\ell_2\le\ell_3$, then either $\ell_3\ge M$ or $\ell_1 \le \frac{1}{M}$.
We consider at first the case where $\ell_3\ge M$. From the perimeter constraint $\ell_1<3/\ell_3$ and $\ell_2<3/\ell_3$ and then, using also that $1/\ell_3^2>0$, we have
\[
F_2(\ell) >
\pi^4 \bigg(\frac{1}{\ell_2^2} +  \frac{1}{\ell_3^2} \bigg) \bigg(\frac{1}{\ell_1^2} + \frac{1}{\ell_3^2} \bigg) > \frac{\pi^4}{\ell_1^2\ell_2^2} > \frac{\pi^4}{81}\ell_3^4.
\]
Recalling that $\ell_3\ge M> 3 \sqrt[4]{12} $ we have 
\[
F_2(\ell) > \frac{\pi^4}{81}M^4>12\pi^4,
\]
which is the value of $F_2$ for $\ell_0$.

In the case where $\ell_1\le\frac{1}{M}$ we simplify $F_2$ to have that
\[
F_2(\ell) > \frac{\pi^4}{\ell_1^4} \ge \pi^4M^4 > 12 \pi^4,
\]
which is the value of $F_2$ in $\ell_0$.

\subsubsection{Third symmetric function}

We consider $\ell\in{\rm Int} K_M$ with $M^4>4\cdot 3^5$. 
Assuming that $\ell$ is a constrained critical point, we have that there exists $\eta\in \R$ such that
\[
\begin{cases}
    &\nabla F_3(\ell) = \eta\, \nabla P(\ell),\\
    &\ell_1\ell_2+\ell_1\ell_3+\ell_2\ell_3 =3 .
\end{cases}
\]
The first component of $\nabla F_3(\ell)=\eta\nabla P(\ell)$ reads
\[
-\frac{ 2\pi^2}{\ell_1^2} \bigg( \Big( \frac{\pi}{\ell_2} \Big)^2 +  \Big( \frac{ \pi}{\ell_3} \Big)^2 \bigg) \bigg( 2\Big( \frac{\pi}{\ell_1} \Big)^2 + \Big( \frac{\pi}{\ell_2} \Big)^2 +  \Big( \frac{ \pi}{\ell_3} \Big)^2 \bigg) 
 =\eta \ell_1(\ell_2+ \ell_3)  = \eta (3 - \ell_2 \ell_3),
\]
from which we deduce that $\eta <0$.
Then, using the notation introduced in \eqref{def abc}
we rewrite the equation as follows
\[
-2 \frac{\pi^6}{\ell_1^2} \frac{\ell_2^2+\ell_3^2}{a^2} \frac{2a^2+b^2+c^2}{abc} +\eta\,a= 3 \eta,
\]
that is, observing that $\ell_1^2 = (bc)/a$  and $\ell_2^2+\ell_3^2=\frac{b^2+c^2}{\ell_1^2}$
\[
- \frac{2\pi^6}{(abc)^3}a^2(b^2+c^2) (2a^2+b^2+c^2) + \eta\,a = 3 \eta .
\]
Repeating the same procedures for the other components of $\nabla F_3(\ell)=\eta\nabla P(\ell)$ we find the system

\[
\begin{cases}
    &- \frac{2\pi^6}{(abc)^3}a^2(b^2+c^2) (2a^2+b^2+c^2) + \eta\,a = 3 \eta,\\
    &- \frac{2\pi^6}{(abc)^3}b^2(a^2+c^2) (a^2+2b^2+c^2) + \eta\,b = 3 \eta,\\
    &- \frac{2\pi^6}{(abc)^3}c^2(a^2+b^2) (a^2+b^2+2c^2) + \eta\,c = 3 \eta .\\
\end{cases}
\]

Matching the first two equations we obtain
\[
- \frac{2\pi^6 }{(abc)^3}\bigg[a^2 (b^2+c^2) (2a^2+b^2+c^2) - b^2 (a^2+c^2) (2b^2+a^2+c^2)\bigg] + \eta\,(a-b) = 0,
\]
which gives
\[
- \frac{2\pi^6 }{(abc)^3} \big[ 2c^2(a^4-b^4) + c^4 (a^2-b^2) + a^2b^2(a^2-b^2) \big]  + \eta\,(a-b) = 0.
\]
If, by contradiction, we assume that $a\neq b$, we can divide by $a-b$ and we have
\[
-\frac{\pi^6 }{(abc)^3} \big[ 2c^2(a^2+b^2)(a+b) + c^4 (a+b) + a^2b^2(a+b) \big]  +\eta = 0,
\]
which implies that $\eta>0$ that is a contradiction. Consequently, $a=b$ and thus $\ell_1=\ell_2$. By studying the other equations we deduce that $\ell_i=1$ for $i=1,2,3$.

Now, if we take a point $\ell \in \R^3_+\setminus {\rm Int}K_M$ and we assume  that $\ell_1\le\ell_2\le \ell_3$, then either $\ell_3\ge M$ or $\ell_1 \le \frac{1}{M}$ (or both).

We consider at first the case where $\ell_3\ge M$. From the perimeter constraint $\ell_1<3/\ell_3$ and $\ell_2<3/\ell_3$.  This gives $\frac{1}{\ell_2^2}+\frac{1}{\ell_3^2}\geq\frac{2}{\ell_2\ell_3}\geq\frac 23$. Then, recalling also that $1/\ell_3^2>0$ we have
\[
    F_3(\ell) >\frac{2\pi^6}{3\ell_1^4}> \frac{2\pi^6}{3^5}M^4.
\]
Since $\ell_3\ge M$ and $M^4>4\cdot 3^5$ we have
\[
F_3(\ell)  >8 \pi^6,
\]
which is the value of $F_3$ in $\ell_0$.

Now we study the case where $\ell_1\le\frac{1}{M}$. Since $\ell_1\le \ell_2\le \ell_3$ we have
\[
    F_3(\ell) >\frac{2\pi^6}{3}M^4,
\]
which is larger than the value of $F_3(\ell_0)=8\pi^6$.

\qed

\section{Proof of Theorem~\ref{main1}}\label{sec:main1}
We start by considering the region
$$
R_{123}=\{(x,y,z)\in\mathbb R^3_+:x\leq y\leq z\}.
$$
If $\ell=(\ell_1,\ell_2,\ell_3)$, setting $x=\ell_1^2$, $y=\ell_2^2$, $z=\ell_3^2$, we have
\begin{align*}
&a(\ell)=g_1(x,y,z):=\pi^2\left(\frac 1y+\frac 1z\right),\\
&b(\ell)=g_2(x,y,z):=\pi^2\left(\frac 1x+\frac 1z\right),\\
&c(\ell)=g_3(x,y,z):=\pi^2\left(\frac 1x+\frac 1y\right),
\end{align*}
and in $R_{123}$ we have
$$
g_1(x,y,z)\leq g_2(x,y,z)\leq g_3(x,y,z).
$$
Recall that the Maxwell's  spectrum  of  $\Omega(\ell)$, in this notation, is given by
\begin{enumerate}[i)]
    \item $\pi^2\left(\frac{k_1^2}{x}+\frac{k_2^2}{y}+\frac{k_3^2}{z}\right)$ counted once, with $k_i\in\mathbb N$, and exactly one $k_i=0$;
    \item $\pi^2\left(\frac{k_1^2}{x}+\frac{k_2^2}{y}+\frac{k_3^2}{z}\right)$  counted twice, with $k_i\in\mathbb N$, $k_i\geq 1$.
\end{enumerate}
Then we have that $\lambda_i(\Omega(\ell))=g_i(x,y,z)$ if and only if 
\begin{equation*}\label{conditions}
\max\left\{\frac 1x+\frac 1y,\frac 1x+\frac 1z,\frac 1y+\frac 1z\right\} \le \min\left\{\frac 4x+\frac 1y,\frac 4x+\frac 1z,\frac 1x+\frac 4y,\frac 4y+\frac 1z,\frac 1x+\frac 4z,\frac 1y+\frac 4z\right\},
\end{equation*}
which is equivalent to the condition
\begin{equation*}\label{conditions2}
   \frac 1x\leq\frac 4z .
\end{equation*}
Hence we have that $\lambda_i(\Omega(\ell))=g_i(x,y,z)$, $i=1,2,3$ if and only if
$$
(x,y,z)\in R_{123}\cap\left\{z\leq 4x\right\}.
$$
We have completed the analysis in the case $0<x\leq y\leq z$. Now, take all permutations of the coordinates $(x,y,z)$. Then we conclude that $\lambda_i(\Omega(\ell))=g_{j(i)}(x,y,z)$ for some permutation $j:\{1,2,3\}\to\{1,2,3\}$ if and only if
$$
(x,y,z)\in\mathcal Q=\{(x,y,z)\in\R^3_+:\, \max\{x,y,z\}\le 4\min\{x,y,z\}\}.
$$
The proof of the theorem is concluded recalling that $x=\ell_1^2$, $y=\ell_2^2$, $z=\ell_3^2$.

\qed

\section{Proof of Theorem~\ref{AutovSingolarmente}}\label{app:A}
Let $\Omega(\ell)$, $\ell=(\ell_1,\ell_2,\ell_3)\in\mathbb R_+^3$, be a cuboid. Once we impose the volume or perimeter constraint, we can write $\ell_3$ as a function of $\ell_1$ and $\ell_2$, and we can understand the eigenvalues as functions of the pair $(\ell_1,\ell_2)\in\R^2_+:=\{(x,y)\in\mathbb R^2:x>0,y>0\}$. Indeed, if we impose the volume constraint $\ell_1\ell_2\ell_3=1$, we can write
\[
\ell_3 = \frac{1}{\ell_1\ell_2},
\]
while if we impose the perimeter constraint $\ell_1\ell_2+\ell_1\ell_3+\ell_2\ell_3=3$ we can write
\[
\ell_3 = \frac{3-\ell_1\ell_2}{\ell_1+\ell_2}.
\]
It is not restrictive to consider cuboids of volume $1$ when considering volume constraint, and cuboids of perimeter $6$ when considering perimeter constraint.
Note that when the perimeter constraint is imposed, we get an additional condition on $\ell_1$ and $\ell_2$: since $\ell_3>0$, necessarily $\ell_1\ell_2<3$.

Thus, the proof of Theorem~\ref{AutovSingolarmente} is recast to the study of suitable functions of two real variables with domain $\mathbb R^2_+$ (volume constraint) or $\mathbb R^2_+\cap\{\ell_1\ell_2<3\}$ (perimeter constraint). The Lipschitz continuity of $\lambda_k(\Omega(\ell))$, $k=1,2,3$, in a neighborhood of $\ell_0=(1,1,1)$ is just a consequence of the fact that in a neighborhood of $\ell_0$ these functions are respectively the minimum, median and maximum of three smooth functions $a(\ell)$, $b(\ell)$, $c(\ell)$ .

\subsection{The first eigenvalue under volume constraint} We start by analyzing the behaviour of the first eigenvalue under volume constraint. As we mentioned, by using the volume constraint $\ell_3=\frac{1}{\ell_1\ell_2}$ we understand $\lambda_1(\Omega(\ell))$ as a function of $\ell_1$ and $\ell_2$ in $\R^2_+$. Recall that the first eigenvalue is given by
$$
    \lambda_1(\Omega(\ell))=\min\left\{a(\ell),b(\ell),c(\ell)\right\},
$$
where $a,b,c$ have been defined in \eqref{symmetric_F}.
Consider the three regions
\begin{eqnarray*}
&&R_1=\{(\ell_1,\ell_2)\in\mathbb R^2_+:\ell_1\leq\min\{\ell_2,\ell_2^{-1/2}\}\},\\
&&R_2=\{(\ell_1,\ell_2)\in\mathbb R^2_+:\ell_2\leq\min\{\ell_1,\ell_1^{-1/2}\}\},\\
&&R_3=\{(\ell_1,\ell_2)\in\mathbb R^2_+:\ell_1\geq\ell_2^{-1/2}{\rm \ and\ }\ell_2\geq\ell_1^{-1/2}\},
\end{eqnarray*}
see Figure~\ref{regions}. We see that $\mathbb R^2_+=R_1\cup R_2\cup R_3$. Then
$$
\lambda_1(\Omega(\ell))=
\begin{cases}
a(\ell)=\pi^2\left(\frac{1}{\ell_2^2}+\ell_1^2\ell_2^2\right) & {\rm if\ }(\ell_1,\ell_2)\in R_1,\\
b(\ell)=\pi^2\left(\frac{1}{\ell_1^2}+\ell_1^2\ell_2^2\right) & {\rm if\ }(\ell_1,\ell_2)\in R_2,\\
c(\ell)=\pi^2\left(\frac{1}{\ell_1^2}+\frac 1{\ell_2^2}\right) & {\rm if\ }(\ell_1,\ell_2)\in R_3.
\end{cases}
$$

\begin{figure}
    \centering
    \includegraphics[width=0.5\textwidth]{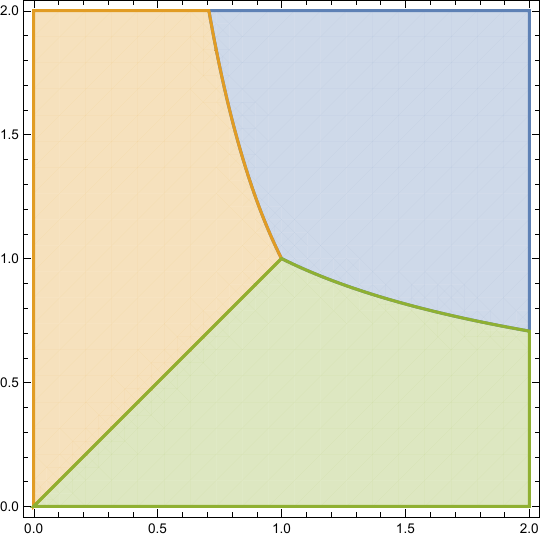}
    \caption{The regions $R_1$ (yellow), $R_2$ (green) and $R_3$ (blue).}
    \label{regions}
\end{figure}

We restrict to the region $R_3$ and consider 
$$
F(\ell_1,\ell_2):=c(\ell)=\pi^2\left(\frac 1{\ell_1^2}+\frac 1{\ell_2^2}\right).
$$
 Fix $c>0$ and consider 
\[
F|_{\{\ell_1=c\}} (\ell_1,\ell_2) =\pi^2 \left( \frac{1}{c^2}+ \frac{1}{\ell_2^2}\right) =:F_c(\ell_2).
\]
Then $F_c$ is strictly decreasing in $\ell_2$ for all $c>0$. Thus, any local maximum of $F$ in $R_3$ is attained on $\partial R_3$. Note that $\partial R_3$  is the disjoint union of two smooth curves intersecting in $(\ell_1,\ell_2)=(1,1)$:
$$
\Gamma_1=\{(\ell_1,\ell_2):\ell_1=\ell_2^{-1/2}\,,\ell_1\leq 1\},\ \ \ \ \ \ \Gamma_2=\{(\ell_1,\ell_2):\ell_2=\ell_1^{-1/2}\,,\ell_1\geq 1\}.
$$
Note also that $\Gamma_1=R_3\cap R_1$ and $\Gamma_2=R_3\cap R_2$. Consider now the restriction of $F$ to $\Gamma_2$, namely the function $G(\ell_1):=\pi^2\left(\frac1{\ell_1^2}+\ell_1\right)$, defined on $[1,+\infty)$. We see that
\begin{itemize}
\item $G$ has a local maximum at the boundary point $\ell_1=1$;
\item $\lim_{\ell_1\to+\infty}G(\ell_1)=+\infty$;
\item $G$ has a local minimum at $\ell_1=2^{1/3}$.
\end{itemize}
The same analysis holds when we restrict $F$ to $\Gamma_1$, just exchanging the roles of $\ell_1$ and $\ell_2$. We conclude that $(1,1)$ is a point of local maximum for $F$ in $R_3$.

See Figure~\ref{F_R3} for the plot of $F$ on its domain of definition $R_3$.
\begin{figure}
    \centering
\includegraphics[width=0.5\textwidth]{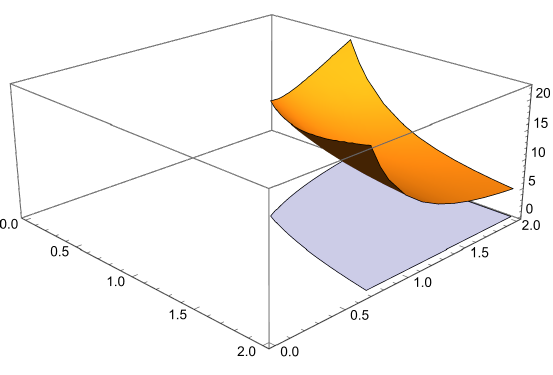}
    \caption{The function $F$ defined in $R_3$.}
    \label{F_R3}
\end{figure}
The above analysis can be carried out in the other two regions $R_1,R_2$ and it actually amounts to a suitable change of variables in $R_3$. Altogether, we deduce that the function
$$
(\ell_1,\ell_2)\mapsto\lambda_1(\Omega(\ell))\,,\ \ \ \ell=\left(\ell_1,\ell_2,\frac 1{\ell_1\ell_2}\right)
$$
defined in $\mathbb R^2_+$ behaves as follows:
\begin{itemize}
\item it is smooth in $\mathbb R^2_+\setminus\{\Gamma_1,\Gamma_2,\Gamma_3\}$, where $\Gamma_3=\{(\ell_1,\ell_2):\ell_1=\ell_2\,,\ell_1\leq 1\}$ (note that $\Gamma_3=R_1\cap R_2$);
\item it has a local maximum at $(\ell_1,\ell_2)=(1,1)$, and the value at this point is $2\pi^2$; here $\lambda_1(\Omega(\ell))$ is not smooth;
\item it has three saddle points at $(\ell_1,\ell_2)=(2^{1/3},2^{-1/6})$, $(2^{-1/6},2^{1/3})$ and $(2^{-1/6},2^{-1/6})$; here $\lambda_1(\Omega(\ell))$ is not smooth;
\item $\sup_{\mathbb R^2_+}\lambda_1(\Omega(\ell))=+\infty$;
\item $\inf_{\mathbb R^2_+}\lambda_1(\Omega(\ell))=0$.
\end{itemize}
See Figure~\ref{first_vol} for a plot of $\lambda_1(\Omega(\ell))$ for $\ell=(\ell_1,\ell_2,\frac 1{\ell_1\ell_2})$ in the region $\mathbb R^2_+$.

\begin{figure}
    \centering
    \includegraphics[width=0.5\textwidth]{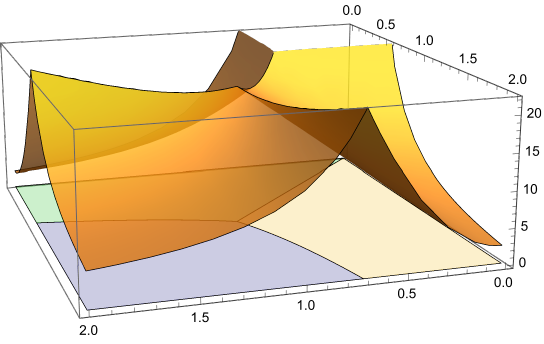}
    \caption{First eigenvalue of cuboids under volume constraint. Precisely, the plot represents $\lambda_1(\Omega(\ell))$ for $\Omega(\ell)=(0,\ell_1)\times(0,\ell_2)\times(0,\frac{1}{\ell_1\ell_2})$. }
    \label{first_vol}
\end{figure}

\medskip

\subsection{The first eigenvalue under perimeter constraint} Next, we consider the perimeter constraint. Hence, we set $\ell_3=\frac{3-\ell_1\ell_2}{\ell_1+\ell_2}$. We consider then the function
$$
(\ell_1,\ell_2)\mapsto\lambda_1(\Omega(\ell))\,,\ \ \ \ell=\left(\ell_1,\ell_2,\frac{3-\ell_1\ell_2}{\ell_1+\ell_2}\right)
$$
restricted to
$$
R=\{(\ell_1,\ell_2)\in\mathbb R^2_+:\ell_1\ell_2<3\}.
$$
Now, $R=R_1\cup R_2\cup R_3$ where
\begin{eqnarray*}
&&R_1=\left\{(\ell_1,\ell_2)\in\mathbb R^2_+:\ell_1\leq\min\left\{\ell_2,\frac{3-\ell_1\ell_2}{\ell_1+\ell_2}\right\}\right\},\\
&&R_2=\left\{(\ell_1,\ell_2)\in\mathbb R^2_+:\ell_2\leq\min\left\{\ell_1,\frac{3-\ell_1\ell_2}{\ell_1+\ell_2}\right\}\right\},\\
&&R_3=\left\{(\ell_1,\ell_2)\in\mathbb R^2_+:0<\frac{3-\ell_1\ell_2}{\ell_1+\ell_2}\leq\min\left\{\ell_1,\ell_2\right\}\right\},
\end{eqnarray*}
see Figure~\ref{regionsP}.
\begin{figure}
    \centering
    \includegraphics[width=0.5\textwidth]{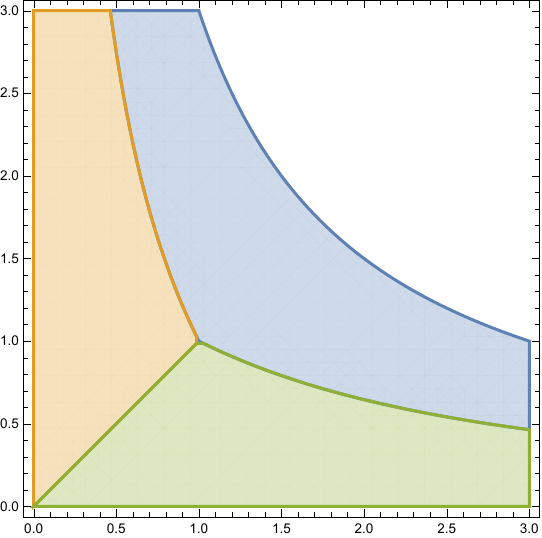}
    \caption{The regions $R_1$ (yellow), $R_2$ (green) and $R_3$ (blue).}
    \label{regionsP}
\end{figure}
We have
$$
\lambda_1(\Omega(\ell))=
\begin{cases}
a(\ell)=\pi^2\left(\frac{1}{\ell_2^2}+\left(\frac{\ell_1+\ell_2}{3-\ell_1\ell_2}\right)^2\right) & {\rm if\ }(\ell_1,\ell_2)\in R_1;\\
b(\ell)=\pi^2\left(\frac{1}{\ell_1^2}+\left(\frac{\ell_1+\ell_2}{3-\ell_1\ell_2}\right)^2\right) & {\rm if\ }(\ell_1,\ell_2)\in R_2;\\
c(\ell)=\pi^2\left(\frac{1}{\ell_1^2}+\frac 1{\ell_2^2}\right) & {\rm if\ }(\ell_1,\ell_2)\in R_3.
\end{cases}
$$
We restrict to $R_3$ and consider  the function
$$
F(\ell_1,\ell_2):=c(\ell)=\pi^2\left(\frac{1}{\ell_1^2}+\frac{1}{\ell_2^2}\right).
$$

Observe that $F$ is smooth in the interior of $R_3$ and, for any fixed $c>0$, the function
\[
 F|_{\{\ell_1=c\}}(\ell_1,\ell_2)=\pi^2\left(\frac{1}{c^2}+\frac{1}{\ell_2^2}\right)=:F_c(\ell_2) 
\]
is strictly decreasing in $\ell_2$. It follows that any local maximum of $F$ in $R_3$ is attained on one of the two curves $\Gamma_1 := R_3 \cap R_1$ or $\Gamma_2 := R_3 \cap R_2$, while  the infimum has to be searched on the curve $\{(\ell_1,\ell_2) \in \mathbb{R}^2_+ : \ell_1 \ell_2 = 3\}$.

In order to find a local maximum of $F$ in $R_3$ we  restrict our analysis to $\Gamma_1 \cup \Gamma_2$. Consider the restriction of $F$ to $\Gamma_1$, namely the function 
\[
G(\ell_1):=F\left(\ell_1,\frac{3-\ell_1^2}{2\ell_1}\right),\quad \ell_1\in(0,1].
\]
We see that $G$ has a local maximum at $\ell_1=1$ and a global minimum at $\ell_1=\bar\ell$, where $\bar\ell \approx 0.906$ is the unique positive root of $5x^6+3x^4+27x^2-27=0$. Moreover, $\lim_{\ell_1\to 0}G=+\infty$. Note that $F|_{\Gamma_2}$ presents the same behaviour, just exchanging the roles of $\ell_1$ and $\ell_2$. 

To determine the infimum of $F$ in $R_3$ we restrict to the curve $\{(\ell_1,\ell_2) \in \R^2_+:\,\ell_1\ell_2=3\}$. If we define $H(\ell_1):=F(\ell_1,\frac 3{\ell_1})$, $\ell_1\in(0,+\infty)$ as the restriction of $F$ to the curve $\ell_1\ell_2=3$, we see that $H$ has a unique global minimum at $\ell_1=\sqrt{3}$ and $H(\ell_1)\to+\infty$ as $\ell_1\to 0$ or $\ell_1\to +\infty$.

The analysis in the regions $R_1,R_2$ is carried out analogously as it can be recast to the analysis in $R_3$ through a change of variables. Altogether we deduce that the function $\lambda_1(\Omega(\ell))$, understood as a function of two variables $(\ell_1,\ell_2)\in R$ after the perimeter constraint is imposed (i.e., $\ell=(\ell_1,\ell_2,\frac{3-\ell_1\ell_2}{\ell_1+\ell_2})$), has the following behaviour:
\begin{itemize}
\item it is smooth in $R\setminus\{\Gamma_1,\Gamma_2,\Gamma_3\}$, where $\Gamma_3=\{(\ell_1,\ell_2):\ell_1=\ell_2,\ell_1\leq 1\}=R_1\cap R_2$;
\item it has a local maximum at $(\ell_1,\ell_2)=(1,1)$, and the value at this point is $2\pi^2$; here $\lambda_1(\Omega(\ell))$ is not smooth;
\item it has three saddle points at $(\ell_1,\ell_2)=(\bar\ell,\frac{3-\bar\ell^2}{2\bar\ell})$, $(\frac{3-\bar\ell^2}{2\bar\ell},\bar\ell)$ and $(\bar\ell,\bar\ell)$; here $\lambda_1(\Omega(\ell))$ is not smooth;
\item it has a positive infimum, which is attained asymptotically as $(\ell_1,\ell_2)\to (0,\sqrt{3})$, $(\sqrt{3},0)$ or $(\sqrt{3},\sqrt{3})$, which are boundary points of $R$; the value of the infimum is $\frac{2\pi^2}{3}$; in this limit, the corresponding cuboid  degenerates to a square of side $\sqrt{3}$;
\item $\sup_R\lambda_1(\Omega(\ell))=+\infty.$
\end{itemize}
See Figure~\ref{first_per} for a plot of $\lambda_1(\Omega(\ell))$ for $\ell=(\ell_1,\ell_2,\frac{3-\ell_1\ell_2}{\ell_1+\ell_2})$ in the region $R$.

\begin{figure}
    \centering
    \includegraphics[width=0.5\textwidth]{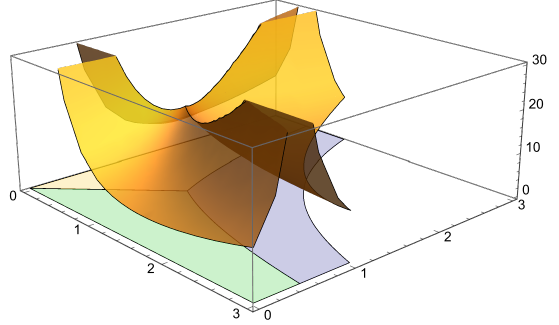}
    \caption{First eigenvalue of cuboids under perimeter constraint. Precisely, the plot represents $\lambda_1(\Omega(\ell))$ for $\Omega(\ell)=(0,\ell_1)\times(0,\ell_2)\times(0,\frac{3-\ell_1\ell_2}{\ell_1+\ell_2})$.}
    \label{first_per}
\end{figure}

\subsection{The second eigenvalue}
In the case of higher eigenvalues, in order to simplify the study, it is convenient to consider the region $0<\ell_1\leq\ell_2\leq\ell_3$. In this region we have
\begin{align*}
&\lambda_1(\Omega(\ell))
=\pi^2\left(\frac{1}{\ell_2^2}+\frac{1}{\ell_3^2}\right),\\
&\lambda_2(\Omega(\ell))
=\pi^2\min\left\{
\frac{1}{\ell_1^2}+\frac{1}{\ell_3^2},
\frac{1}{\ell_2^2}+\frac{4}{\ell_3^2}
\right\},\\
&\lambda_3(\Omega(\ell))
=\pi^2\min\Biggl\{
\frac{1}{\ell_1^2}+\frac{1}{\ell_2^2},
\max\left\{
\frac{1}{\ell_1^2}+\frac{1}{\ell_3^2},
\frac{1}{\ell_2^2}+\frac{4}{\ell_3^2}
\right\},
\frac{4}{\ell_2^2}+\frac{1}{\ell_3^2},
\frac{1}{\ell_2^2}+\frac{9}{\ell_3^2}
\Biggr\}.
\end{align*}

We consider here only the nature of the point $\ell_0=(1,1,1)$ for both volume and perimeter constraints.

Consider first the volume constraint $\ell_1\ell_2\ell_3=1$; then the second eigenvalue is given by the function
$$
F(\ell_1,\ell_2)=\pi^2\min\left\{\left(\frac{1}{\ell_1^2}+\ell_1^2\ell_2^2\right),\left(\frac{1}{\ell_2^2}+4\ell_1^2\ell_2^2\right)\right\}
$$
in the region $R=\{(\ell_1,\ell_2):0<\ell_1\leq\ell_2\leq\frac 1{\ell_1\ell_2}\}$. Let $\eps>0$  be sufficiently small such that
\[
F(\ell_1,\ell_2)=\pi^2\left(\frac 1{\ell_1^2}+\ell_1^2\ell_2^2\right) \quad \text{on } R\cap \overline{B_{\eps}},
\]
where $B_{\eps}=\{(\ell_1,\ell_2):|\ell_1-1|+|\ell_2-1|<\eps\}$. Restricting $F$ to the segment  $\ell_1=\ell_2$ in $R\cap B_{\eps}$, we see that the resulting one-variable function $\pi^2\left(\frac 1{\ell_1^2}+\ell_1^4\right)$ defined in $(1-\eps,1]$ is strictly increasing near $\ell_1=1$. Restricting $F$ to the curve $\ell_2=\ell_1^{-1/2}$, the resulting one-variable function $\pi^2\left(\frac{1}{\ell_1^2}+\ell_1\right)$ defined in $(1-c\eps,1]$ (for a suitable $c>0$) is strictly decreasing near $\ell_1=1$. Hence, $(1,1)$ is neither a local maximum nor a local minimum for $F$ restricted to $R \cap B_\varepsilon$. Recall that we have only considered the region $0<\ell_1\leq\ell_2\leq\frac 1{\ell_1\ell_2}$, but the same analysis can be carried out for the other five regions defined by $0<\ell_{i_1}\leq\ell_{i_2}\leq\ell_{i_3}$, where $\{i_1,i_2,i_3\}$ is a permutation of $\{1,2,3\}$. Nevertheless, just observing the behaviour in $R$ we can conclude that $\ell_0=(1,1,1)$ is not a local extremum for $\lambda_2(\Omega(\ell))$ under volume constraint. See Figure~\ref{second}.  

The same analysis can be carried out for the perimeter constraint. 

\begin{figure}
    \centering
\includegraphics[width=0.5\textwidth]{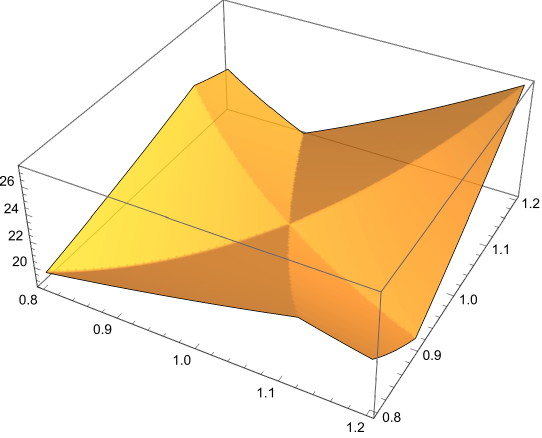}
    \caption{ Second eigenvalue of cuboids under volume constraint. Precisely, the plot represents $\lambda_2(\Omega(\ell))$ for $\Omega(\ell)=(0,\ell_1)\times(0,\ell_2)\times(0,\frac 1{\ell_1\ell_2})$ in a neighborhood of $(\ell_1,\ell_2)=(1,1)$.}
    \label{second}
\end{figure}

\medskip

\subsection{The third eigenvalue} Also for the third eigenvalue, in order to simplify the analysis, it is convenient to restrict to the case $0<\ell_1\le\ell_2\le \ell_3$. As for the second eigenvalue, we analyze only the nature of the point $\ell_0 =(1,1,1)$ under both volume and perimeter constraint. Consider first the volume constraint $\ell_1\ell_2\ell_3=1$; then the third eigenvalue is given by the function
\begin{multline*}
F(\ell_1,\ell_2) =
\pi^2\min\Bigg\{
\left(\frac{1}{\ell_1^2}+\frac{1}{\ell_2^2}\right),
\max\left\{\left(\frac{1}{\ell_1^2}+\ell_1^2\ell_2^2\right),\left(\frac{1}{\ell_2^2}+ 4\ell_1^2\ell_2^2\right)\right\},\\\left(\frac{4}{\ell_2^2}+\ell_1^2\ell_2^2\right),\left(\frac{1}{\ell_2^2}+\frac{9}{\ell_2^2}\right)
\Bigg\}
\end{multline*}
in the region $R=\{(\ell_1,\ell_2):0<\ell_1\leq\ell_2\leq\frac 1{\ell_1\ell_2}\}$, see Figure~\ref{region3}. Let $\eps>0$  be sufficiently small such that
\[
F(\ell_1,\ell_2)=\pi^2\left(\frac{1}{\ell_1^2}+\frac{1}{\ell_2^2}\right) \quad \text{on } R\cap \overline{B_\eps},
\]
where $B_{\eps}=\{(\ell_1,\ell_2):|\ell_1-1|+|\ell_2-1|<\eps\}$. 
If we restrict $F$ to horizontal lines in $R \cap B_\eps$ (i.e. we consider $F|_{\{\ell_2=c\}}$ for some $c>0$), we have
\[
F|_{\{\ell_2=c\}} (\ell_1,\ell_2) =\pi^2 \left( \frac{1}{\ell_1^2}+ \frac{1}{c^2}\right) =:F_c(\ell_1),
\]
which is a strictly decreasing function. Thus, its minimum in $R\cap \overline{B_\eps}$ is attained on $\partial R\cap \overline{B_\eps}$. This proves that the minimum of $F$ in $R\cap \overline{B_\eps}$ is reached on  $\partial R\cap \overline{B_\eps}$. Note that $\partial R\cap\{\ell_1>0\}$ is the disjoint union of two smooth curves intersecting in $(\ell_1,\ell_2)=(1,1)$:
$$
\Gamma_1'=\{(\ell_1,\ell_2):\ell_2=\ell_1^{-1/2}\,,\ell_1\leq 1\},\ \ \ \ \ \ \Gamma_2'=\{(\ell_1,\ell_2):\ell_1=\ell_2\,,\ell_1\leq 1\}.
$$
Consider now the restriction of $F$ to $\Gamma_1'$, namely the function
\[
G(\ell_1):=\pi^2\left(\frac1{\ell_1^2}+\ell_1\right), \quad \ell_1\in(0,1].
\]
We see that $G$ has a minimum at the boundary point $\ell_1=1$ and $\lim_{\ell_1\to 0}G(\ell_1)=+\infty$.
The same analysis holds when we restrict $F$ to $\Gamma_2'$: $F|_{\Gamma_2'}$ has a minimum at $\ell_2=1$ and diverges to $+\infty$ as $\ell_2\to 0$. Hence $F$ restricted to $R\cap\overline{B_\eps}$ attains its minimum at $(1,1)$. Recall that we have only considered the region $0<\ell_1\le\ell_2\le\ell_3$, but the same analysis can be carried out for the other five regions defined by $0<\ell_{i_1}\leq\ell_{i_2}\leq\ell_{i_3}$, where $\{i_1,i_2,i_3\}$ is a permutation of $\{1,2,3\}$. Then the function
\[
(\ell_1,\ell_2)\mapsto\lambda_3(\Omega(\ell))\,,\ \ \ \ell=\left(\ell_1,\ell_2,\frac 1{\ell_1\ell_2}\right)
\]
has a local minimum at $(\ell_1,\ell_2)=(1,1)$, where it is not smooth, see Figure~\ref{third}.

 The same analysis can be carried out for the perimeter constraint.

\begin{figure}
    \centering
    \includegraphics[width=0.5\textwidth]{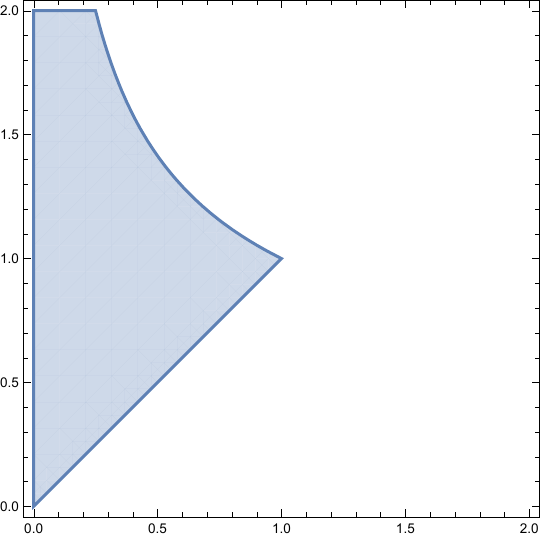}
    \caption{The region $R$}
    \label{region3}
\end{figure}

\begin{figure}
    \centering
\includegraphics[width=0.5\textwidth]{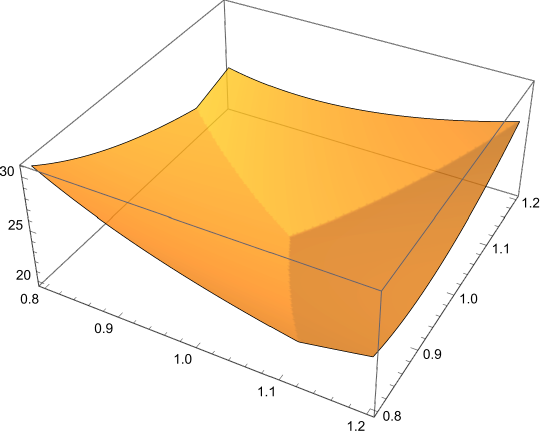}
    \caption{ Third eigenvalue of cuboids under volume constraint. Precisely, the plot represents $\lambda_3(\Omega(\ell))$ for $\Omega(\ell)=(0,\ell_1)\times(0,\ell_2)\times(0,\frac 1{\ell_1\ell_2})$ in a neighborhood of $(\ell_1,\ell_2)=(1,1)$.}
    \label{third}
\end{figure}

This concludes the proof. \qed

\begin{figure}
    \centering
\includegraphics[width=0.5\textwidth]{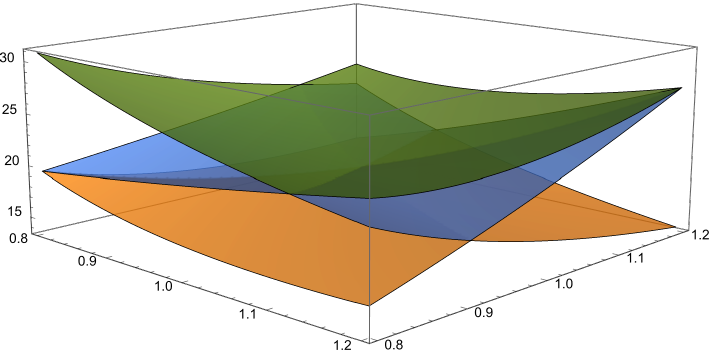}
    \caption{First three eigenvalues of cuboids under volume constraint. Precisely, the plot represents $\lambda_i(\Omega(\ell))$, $i=1,2,3$, for $\Omega(\ell)=(0,\ell_1)\times(0,\ell_2)\times(0,\frac 1{\ell_1\ell_2})$ in a neighborhood of $(\ell_1,\ell_2)=(1,1)$.}
    \label{alltogether}
\end{figure}

\section{Proof of Theorem ~\ref{prop0}}\label{app:B}

 In order to prove Theorem~\ref{prop0}, it is sufficient to exhibit a domain with the same volume/perimeter as the unit ball for which the third eigenvalue is smaller than that of the ball. Throughout this section we denote by $B$ the unit ball in $\mathbb R^3$.

{\bf Idea of the proof.} The proof can be sketched as follows: we consider a small domain $\omega$, of diameter $2\eps$, homeomorphic to a ball, with a small first Maxwell's eigenvalue. We attach this domain to $B$ through a thin cylinder isometric to $D_\eta\times(0,s)$, where $D_\eta\subset\mathbb R^2$ is a disk of radius $\eta>0$, and $s>0$. If the parameters are chosen in a suitable way, the first Maxell's eigenvalue of the resulting domain is smaller than the first Maxwell's eigenvalue of $\omega$. To obtain that the first $N$ eigenvalues are small, we attach $N$ such `handles' to $B$. Now we are ready to detail the steps of the proof.

\smallskip

{\bf Step 1. The small domain with small eigenvalue.} Given the parameters $0<h<\delta<\eps$ we define the following spherical shell with a hole
    \begin{equation} 
    \label{def Omega eps,delta,h}
        \omega_{\eps,\delta,h} := \left\{ (1-t) \rho :\, \rho \in \BBS^2_\eps \setminus \mathcal D_\delta,\;t \in (0,h)  \right\},
    \end{equation}
where $\BBS^2_\eps$ denotes the sphere of radius $\eps$ and $\mathcal D_\delta$ denotes the closure of a geodesic disc (spherical cap) of radius $\delta$ on $\BBS^2_\eps$. By construction, $\omega_{\eps,\delta,h}$ has diameter $2\eps$. We prove in Lemma~\ref{appendix B, lemma 1} that, for any $\eps>0$,
$$
\lim_{h,\delta\to 0}\lambda_1(\omega_{\eps,\delta,h})=0.
$$

\smallskip
{\bf Step 2. The cylinder.}
Let $0<\eta<s$ and define the cylinder
\begin{equation}\label{cilinder}
C_{\eta,s}:=D_{\eta}\times(0,s),
\end{equation}
where $D_{\eta}\subset\mathbb R^2$ is a disk of radius $\eta$.

\smallskip 

{\bf Step 3. The dumbbell domain}

Consider now the parameters $0<h<\delta<\eps<1$ and $0<\eta<\eps$. We can now introduce the domain on which is based the construction of $\Omega_{N,\eps}$  for $N=1$. It is a `dumbbell-like' domain
$\Omega= \Omega(\eps,\delta,h,\eta,s)$ defined by 
$$
\Omega(\eps,\delta,h,\eta,s):=B\cup C_{\eta,s}\cup\omega_{\eps,\delta,h}.
$$
See Figure~\ref{fig1} for a representation of $\Omega(\eps,\delta,h,\eta,s)$. 

\begin{figure}
    \centering
\includegraphics[width=0.7\textwidth]{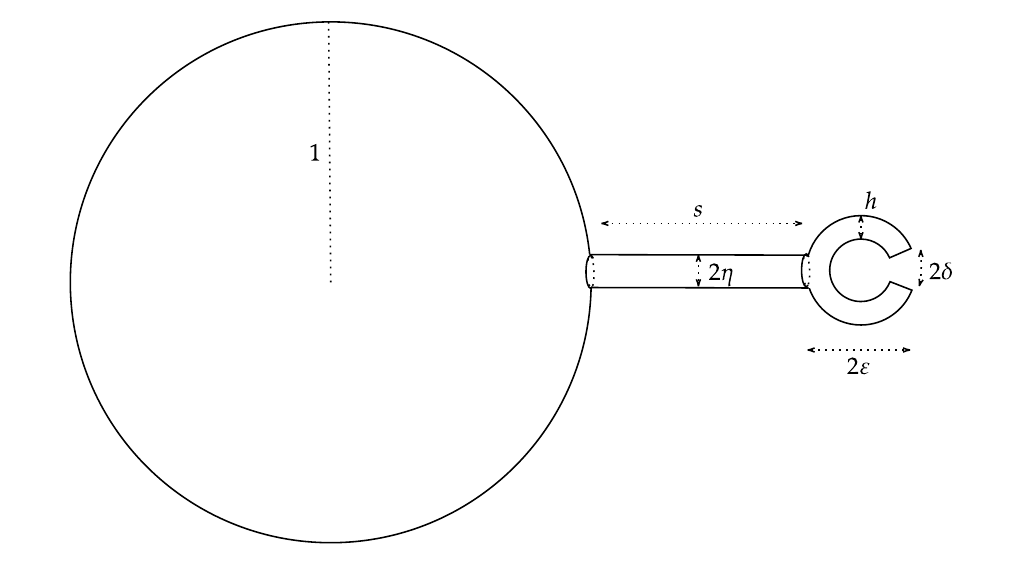}
    \caption{The dumbbell with one handle}
    \label{fig1}
\end{figure}

\smallskip

{\bf Step 4. The first eigenvalue of $\Omega$ is bounded in terms of that of $\omega_{\eps,\delta ,h}$.} Fix $\eps>0$ and  let $\Omega=\Omega(\eps,\delta,h,\eta,s)$ as above. We prove here that $\lambda_1(\Omega)$ can be bounded from above in terms of $\lambda_1(\omega_{\eps,\delta,h})$ for a suitable choice of the parameters $\eta,s$.

To do so, we write the dumbbell in an appropriate system of Cartesian coordinates $(x,y,z)$. In these coordinates, $B=\{(x,y,z)\in\mathbb R^3:x^2+y^2+z^2<1\}$ and
    \[
    C_{\eta,s} = \left\{(x,y,z)\in\R^3:\,x \in\left(\sqrt{1-\eta^2},\,s+ \sqrt{1-\eta^2}\right),\;y^2+z^2< \eta^2\right\}.
    \]
We center the sphere $\BBS^2_\eps$ defining $\omega_{\eps,\delta,h}$ at $(s+\sqrt{1-\eta^2}+\sqrt{\eps^2-\eta^2},0,0)$.  Recall that by construction $0<h<\delta<\eps<1$ and $0<\eta<\eps$, see Figure~\ref{fig1}.
Let also
$$
P=(s+\sqrt{1-\eta^2},0,0)
$$
be the point of intersection of $\omega_{\eps,\delta,h}$ with the $x$-axis. We denote by $X$ a point of coordinates $(x,y,z)$.

We define the function
    \[
    \phi(r) := 
    \begin{cases}
0\,, & {\rm if\ }r<1\,\\
r-1\,, & {\rm if\ }1\leq r<2\,,\\
1\,, & {\rm if\ }r\geq 2\,,
 \end{cases}
    \qquad \phi_\eta(r):= \phi\left( \frac{r}{\eta} \right),
    \]
and the corresponding cut-off
    \[
    \phi_{\eta}(X):= \phi_\eta(|X-P|).
    \]
    Let $u$ be an eigenfunction  associated with $\lambda_1(\omega_{\eps,\delta,h})$, normalized by $||u||_{L^2(\omega_{\eps,\delta,h})}=1$.
    Let
    \[
    u_{\eta}:=
    \begin{cases}
u\phi_\eta & {\rm in\ }\omega_{\eps,\delta,h}\,,\\
0 & {\rm in\ }\Omega\setminus\overline{\omega_{\eps,\delta,h}}.
\end{cases}
    \]
  The function $u_{\eta}$ satisfies the boundary condition $u_{\eta}\times \nu = 0$ on $\partial \Omega$ by construction. Now we would like to use $u_{\eta}$ as test field in the Rayleigh quotient defining $\lambda_1(\Omega)$. However, this is not directly possible since the function $u_{\eta}$ is not necessarily divergence free. To bypass this, we introduce the following eigenvalue problem    
    \begin{equation}
        \label{hodge Pb}
   \begin{cases}
{\rm curl\,curl\,}u-\nabla{\rm div\,}u=\lambda^Hu\,, & {\rm in\ }\Omega\,,\\
u\times\nu =0\,, & {\rm on\ }\partial\Omega\,,\\
{\rm div\,}u=0\,, & {\rm on\ }\partial\Omega\,.
   \end{cases}
    \end{equation}
Problem \eqref{hodge Pb} is the formulation in terms of vector fields in $\mathbb R^3$ of the eigenvalue problem for the Hodge Laplacian acting on $1$-forms with relative boundary conditions. See e.g., \cite{guerini_savo} for an introduction to the Hodge-Laplacian spectrum on manifolds with boundary (see also \cite{ferraresso_provenzano}). In particular, the spectrum of \eqref{hodge Pb} is given by the union of the spectrum of \eqref{curlcurl} on $\Omega$, that is, the set of Maxwell's eigenvalues, and the spectrum of the Dirichlet Laplacian on $\Omega$. We denote by
$$
0\leq\lambda_1^H(\Omega)\leq\lambda_2^H(\Omega)\leq\cdots\leq\lambda_j^H(\Omega)\leq\nearrow+\infty
$$
the eigenvalues of \eqref{hodge Pb}. Since $\Omega$ is homeomorphic to a ball, we have $\lambda_1^H(\Omega)>0$. The eigenvalues of \eqref{hodge Pb} are characterized variationally by    
    \begin{equation*}
        \lambda_j^H(\Omega)= \inf_{\stackrel{V \subset H_0({\rm curl},\Omega)\cap H({\rm div},\Omega)}{{\rm dim}V=j}} \sup_{v\in V\setminus\{0\}} \frac{\int_\Omega |{\rm curl }\, v|^2 + \int_\Omega|{\rm div}\,v|^2}{\int_\Omega |v|^2}.
    \end{equation*}
  
  We bound from above $\lambda_1^H(\Omega)$ using $u_{\eta}$ as test function, which is admissible:
\begin{equation}\label{appendix B, dim lemma 2, eq1}
\lambda_1^H(\Omega)\leq \frac{\int_{\Omega}|{\rm curl}\,u_\eta|^2+\int_{\Omega}|{\rm div}\,u_\eta|^2}{\int_\Omega|u_\eta|^2}.
\end{equation}
We have
   \begin{equation}\label{est000}
    \begin{split}
        &\int_\Omega |{\rm curl }\, u_{\eta}|^2\\&=\int_{\omega_{\eps,\delta,h}} |{\rm curl }\, u|^2 - \int_{\omega_{\eps,\delta,h}\cap B_{2\eta}} |{\rm curl }\, u|^2 + \int_{\omega_{\eps,\delta,h}\cap B_{2\eta}} |{\rm curl }\, (u\, \phi_{\eta})|^2\\&=
\lambda_1(\omega_{\eps,\delta,h}) - \int_{\omega_{\eps,\delta,h}\cap B_{2\eta}} |{\rm curl }\, u|^2 + \int_{\omega_{\eps,\delta,h}\cap B_{2\eta}} |{\rm curl }\, (u\, \phi_{\eta})|^2\, ,  \end{split}
  \end{equation}
    where $B_{2\eta}$ denotes the ball of radius $2\eta$ centered at $P$. An explicit estimate gives
    $$
    |{\rm curl}\,(u\phi_\eta)|^2-|{\rm curl}\,u|^2\leq 2|u|^2|\nabla\phi_\eta|^2+|\phi_\eta|^2|{\rm curl}\,u|^2\,.
    $$
   Then
\begin{multline} \label{appendix B, lemma2, dim, stima 1}
\int_{\omega_{\eps,\delta,h}\cap B_{2\eta}}|{\rm curl}\,(u\phi_\eta)|^2-|{\rm curl}\, u|^2\\\leq\frac{32}{3}\pi \eta\left(2\|u\|_{L^{\infty}(\omega_{\eps,\delta,h})}^2+\|{\rm curl}\,u\|_{L^{\infty}(\omega_{\eps,\delta,h})}^2\eta^2\right).
\end{multline}
We now estimate the divergence term. Since ${\rm div} u=0$ on $\omega_{\eps,\delta,h}$,
    \begin{equation}
    \label{appendix B, lemma2, dim, stima 2}
    \begin{split}
        \int_\Omega |{\rm div }\, u_{\eta}|^2& = \int_{\omega_{\eps,\delta,h}\cap B_{2\eta}} |{\rm div }\, (u\phi_{\eta})|^2\\& \le ||u||^2_{L^\infty(\omega_{\eps,\delta,h})} ||\nabla \phi_{\eta}||^2_{L^\infty(B_{2\eta})}|B_{2\eta}| \leq\frac{32}{3}\pi \,\eta ||u||^2_{L^\infty(\omega_{\eps,\delta,h})}.
    \end{split}
    \end{equation}
    Now we study the denominator in \eqref{appendix B, dim lemma 2, eq1}
    \[
    \begin{split}
        \int_\Omega | u_{\eta}|^2 &= \int_{\omega_{\eps,\delta,h}} | u|^2- \int_{\omega_{\eps,\delta,h}\cap B_{2\eta}} | u|^2 + \int_{\omega_{\eps,\delta,h}\cap B_{2\eta}} |u\, \phi_{\eta}|^2\\&
    =1 - \int_{\omega_{\eps,\delta,h}\cap B_{2\eta}} | u|^2 + \int_{\omega_{\eps,\delta,h}\cap B_{2\eta}} |u\, \phi_{\eta}|^2.
    \end{split}
    \]
  We estimate
    \begin{multline}
    \label{appendix B, lemma2, dim, stima 3}
        \int_{\omega_{\eps,\delta,h}\cap B_{2\eta}} |u\, \phi_{\eta}|^2 - | u|^2 \\
        \ge -||u||^2_{L^{\infty}(\omega_{\eps,\delta,h})}  || 1-\phi_{\eta}^2||^2_{L^\infty(B_{2\eta})}\, |B_{2,\eta}| \ge -\frac{32}{3}\pi \eta^3 ||u||^2_{L^{\infty}(\omega_{\eps,\delta,h})}.
    \end{multline}
    Using \eqref{est000},\eqref{appendix B, lemma2, dim, stima 1}, \eqref{appendix B, lemma2, dim, stima 2}, \eqref{appendix B, lemma2, dim, stima 3} in \eqref{appendix B, dim lemma 2, eq1} we obtain
    \begin{equation}
\lambda_1^H(\Omega) \leq\frac{\lambda_1(\omega_{\eps,\delta,h})+\frac{32}{3}\pi\eta\left(3\|u\|_{L^{\infty}(\omega_{\eps,\delta,h})}^2+\|{\rm curl}\,u\|_{L^{\infty}(\omega_{\eps,\delta,h})}^2\eta^2\right)}{1-\frac{32}{3}\pi\eta^3\|u\|_{L^{\infty}(\omega_{\eps,\delta,h})}^2},
    \end{equation}
     provided that $\eta$ is sufficiently small.
Now, since $\|u\|_{L^2(\omega_{\eps,\delta,h})}=1$, the norms $\|u\|_{L^{\infty}(\omega_{\eps,\delta,h})}$ and $\|{\rm curl}\,u\|_{L^{\infty}(\omega_{\eps,\delta,h})}$ depend only on $\eps,h,\delta$, and once they have been fixed, we can choose $\eta=\eta(\eps,h,\delta)$ small enough such that
$$
\lambda_1^H(\Omega)\leq (1+\eps)\lambda_1(\omega_{\eps,\delta,h}).
$$
Lemma~\ref{appendix B, lemma 1} ensures that $\delta,h$ can be chosen sufficiently small such that 
$$
\lambda_1(\omega_{\eps,\delta,h})\leq\frac{\eps}{2(1+\eps)}.
$$
Hence
$$
\lambda_1^H(\Omega)\leq\frac{\eps}{2}.
$$
Now, $\lambda_1^H(\Omega)=\min\{\lambda_1(\Omega),\lambda_1^D(\Omega)\}$, where $\lambda_j^D(\Omega)$ denote the Dirichlet eigenvalues of the Laplacian on $\Omega$. A result of Rohleder \cite{rohleder} states that $\lambda_3(\Omega)\leq\lambda_1^D(\Omega)$, hence $\lambda_j^H(\Omega)=\lambda_j(\Omega)$ for $j=1,2,3$. We have proved that
$$
\lambda_1(\Omega)\leq\frac{\eps}{2}.
$$

\smallskip
 {\bf Step 5. Volume/perimeter constraint and Hausdorff distance}

In the previous step, for any $\eps>0$ (small) we have produced an example of a domain $\Omega=\Omega(\eps,\delta,h,\eta,s)$ such that
\begin{equation}\label{ineq_part}
\lambda_1(\Omega)\leq\frac{\eps}{2}.
\end{equation}
 The parameters $\delta,h,\eta$ depend on $\eps$. Moreover, we can always take $\eta\leq\eps^2$. Up to now, the parameter $s$ had no true influence in the validity of \eqref{ineq_part}, and can be any fixed positive number. Take $s=\eps$. Now we see that
 $$
|\Omega|=\frac{4}{3}\pi+O(\eps^3)\,,\ \ \ |\partial\Omega|=4    \pi+O(\eps^2).
 $$
This implies that $\frac{|\Omega|}{|B|}=1+O(\eps^3)$, and if $\eps$ is small enough, we can assume that
 $$
\frac{|\Omega|}{|B|}\leq 2, \quad  \frac{|\partial\Omega|}{|\partial B|}\leq 2.
$$
When considering the volume constraint, set
$$
\Omega_{1,\eps}=\frac{|B|^{1/3}}{|\Omega|,^{1/3}}\Omega,
$$
and then $|\Omega_{1,\eps}|=|B|$ and
$$
\lambda_1(\Omega_{1,\eps})=\frac{|\Omega|^{2/3}}{|B|^{2/3}}\lambda_1(\Omega)<\eps.
$$
When considering the perimeter constraint, set
$$
\Omega_{1,\eps}=\frac{|\partial B|^{1/2}}{|\partial\Omega|^{1/2}}\Omega,
$$
and then $|\partial\Omega_{1,\eps}|=|\partial B|$ and
$$
\lambda_1(\Omega_{1,\eps})=\frac{|\partial\Omega|}{|\partial B|}\lambda_1(\Omega)\leq \eps.
$$
From the choice $s=\eps$, and from the fact that $\frac{|\Omega|}{| B|}=1+O(\eps^3)$, $\frac{|\partial\Omega|}{|\partial B|}=1+O(\eps^2)$, we have that $d_H(B,\Omega_{1,\eps})\leq 3\eps$.

\smallskip

{\bf Step 6. Construction of a domain with $N$ small eigenvalues.}
To prove the result with any $N\geq 1$, it is sufficient to attach $N$ congruent copies of the handle $C_{\eta,s}\cup\omega_{\eps,\delta,h}$ as in the previous step at $N$ distinct points of $\partial B$. Precisely, define
$$
\Omega=B\cup(C_{\eta,s}^1\cup\omega_{\eps,\delta,h}^1)\cup\cdots\cup (C_{\eta,s}^N\cup\omega_{\eps,\delta,h}^N),
$$
where the parameters are taken sufficiently small such that $(C_{\eta,s}^i\cup\omega_{\eps,\delta,h}^i)\cap(C_{\eta,s}^j\cup\omega_{\eps,\delta,h}^j)=\emptyset$ if $i\ne j$.

 Associated to each handle $(C_{\eta,s}^i\cup\omega_{\eps,\delta,h}^i)$, define a test function $u_{\eta,i}$, $i=1,...,N$ as in Step 4. By construction, these functions $u_{\eta,i}$ are disjointly supported, hence 
$$
V=\left\{\sum_{i=1}^N{a_iu_{i,\eta}}:a_i\in\mathbb R\right\}
$$
is a $N$-dimensional subspace of $H_0({\rm curl}\,,\Omega)\cap H({\rm div}\,,\Omega)$. Thus
\begin{align*}
\lambda_N^H(\Omega)&\leq\sup_{v\in V\setminus\{0\}}\frac{\int_{\Omega}|{\rm curl}\,v|^2+|{\rm div}\,v|^2}{\int_{\Omega}|v|^2}\\&
=\sup_{(a_1,...,a_N)\in\mathbb R^N\setminus\{0\}}\frac{\sum_{i=1}^Na_i^2\left(\int_{\Omega} |{\rm curl}\,u_{i,\eta}|^2+|{\rm div}\,u_{i,\eta}|^2\right)}{\sum_{i=1}^Na_i^2\int_{\Omega} |\,u_{i,\eta}|^2}\\&=\frac{\int_{\Omega} |{\rm curl}\,u_{i,\eta}|^2+|{\rm div}\,u_{i,\eta}|^2}{\int_{\Omega} |\,u_{i,\eta}|^2}
\end{align*}
since the Rayleigh quotients of each $u_{i,\eta}$ are the same  because the handles are congruent.
Choosing the parameters as in Steps 1-5 we find that
$$
\lambda_N^H(\Omega)\leq\frac{\eps}{2}.
$$
Here, for $N\geq 3$ we cannot deduce immediately from \cite{rohleder} that $\lambda_N^H(\Omega)=\lambda_N(\Omega)$. However, since it is well-known that $\lambda_1^D(\Omega)\geq \frac{C}{|\Omega|^{2/3}}>0$  for some universal $C>0$, we deduce that if $\eps$ is small enough, then $\lambda_N^H(\Omega)\leq\frac{\eps}{2}$ implies $\lambda_N^H(\Omega)=\lambda_N(\Omega)$. The discussion on the volume/perimeter renormalization and on the Hausdorff distance is now exactly the same as in Step 5.

\smallskip

Note that from the proof we have $d_H(B,\Omega)<3\eps$. However, to get the statement of Theorem~\ref{prop0} it is enough to set $\eps'=3\eps$. This concludes the proof.
\qed

\begin{lem}\label{appendix B, lemma 1}
    Let $0<h<\delta<\eps$ and define $\omega_{\eps,\delta,h}$ as in \eqref{def Omega eps,delta,h}. Then
        \begin{equation*}
            \lim_{\delta\to 0}\lim_{h\to 0}\lambda_1(\omega_{\eps,\delta,h})=0.
        \end{equation*}
\end{lem}

\begin{proof}
Let $t>0$. It is known (see e.g., \cite{COURTOIS}) that there exists $\delta=\delta(t)>0$ such that
    \[
\lambda_1^\Delta(\BBS^2_\eps\setminus \mathcal D_\delta)<\frac{t}{2},
    \]
    where $\lambda_1^\Delta\left(\BBS^2_\eps\setminus\mathcal D_\delta\right)$ is the first Dirichlet eigenvalue of the Laplace-Beltrami operator on $\BBS^2_\eps\setminus\mathcal D_\delta$.

Now, from \cite[ Theorem 1.1]{ferraresso_provenzano} we have that there exists $h=h(t,\delta)>0$ such that
    \[
    \left| \lambda_1\left(\omega_{\eps,\delta,h}\right) -\lambda_1^\Delta\left(\BBS^2_\eps\setminus \mathcal D_\delta\right) \right|<\frac{t}{2}.
    \]
    This concludes the proof.
\end{proof}
As a consequence of the previous results, the ball is not a local minimiser for $\lambda_1(\Omega)$ under either a volume or a perimeter constraint. Moreover, it is not a local minimiser for any of the three elementary symmetric functions of the first three Maxwell eigenvalues under either constraint.

Note that in the previous construction we can replace $B$ with any other smooth domain $\Omega$ homeomorphic to a ball. Hence we deduce the following

\begin{cor}
There are no local Lipschitz minimisers homeomorphic to a ball for $\lambda_1(\Omega)$ under either a volume or a perimeter constraint. The same holds true for elementary symmetric functions of the first three eigenvalues.
\end{cor}

\subsection{Remarks on the local maximality of $\lambda_1$} We end this appendix discussing the case of the maximisation of the first eigenvalue. Let $\Omega_1,\Omega_2$ be two fixed bounded domains, and let $C_{\eta,s}$ be a cylinder as above. Define the dumbbbell domain $\Omega_{\eta,s}=\Omega_1\cup C_{\eta,s}\cup\Omega_2$ and assume that it is Lipschitz for all $\eta>0$ small, see Figure~\ref{fig2}.

\medskip

{\bf Assumption.} We assume, without providing here a proof, that
\begin{equation}\label{conv_1}
\lim_{\eta\to 0}\lambda_j(\Omega_{\eta,s})=\lambda_j(\Omega_1\cup\Omega_2),
\end{equation}
where the union $\Omega_1\cup\Omega_2$ is disjoint. 

\medskip

The Maxwell's spectrum of $\Omega_1\cup\Omega_2$ is just the union of the spectra of $\Omega_1$ and $\Omega_2$. Assume that \eqref{conv_1} is true. Fix $\eps>0$, take $\Omega_1=B$, $\Omega_2=B_\eps$ and $s=\eps$, where $B_\eps$ is a ball of radius $\eps<1$. From \eqref{conv_1} we deduce that there exists $0<\eta=\eta(\eps)<\eps$ such that
\begin{equation}
\label{appendix b, conroes 2, conclusione, eq1}
    \lambda_1(B) \leq \lambda_1(\Omega_{\eta(\eps),\eps}) + C\eps^4,
\end{equation}
where $C>0$ is some positive constant. We contract $\Omega_{\eta(\eps),\eps}$ to find a domain with the same volume or perimeter as $B$. Precisely, in the case of volume constraint set
$$
\Omega_\eps:=\frac{|B|^{1/3}}{|\Omega_{\eta(\eps),\eps}|^{1/3}}\Omega_{\eta(\eps),\eps}.
$$
Then \eqref{appendix b, conroes 2, conclusione, eq1} becomes
\begin{equation}
\label{appendix b, conroes 2, conclusione, eq2}
\frac{|\Omega_{\eta(\eps),\eps}|^{2/3}}{|B|^{2/3}}\lambda_1(B)\leq\lambda_1(\Omega_\eps)+C\eps^4\frac{|\Omega_{\eta(\eps),\eps}|^{2/3}}{|B|^{2/3}}.
\end{equation}
Now, since $s=\eps$ and $\eta<\eps$, we have
\begin{equation}\label{comp}
1+C_1\eps^3\leq\frac{|\Omega_{\eta(\eps),\eps}|^{2/3}}{|B|^{2/3}}\leq 1+C_2\eps^3,
\end{equation}
where $C_1,C_2>0$ are constant independent of $\eps$.
Then \eqref{appendix b, conroes 2, conclusione, eq2} can be re-written as
$$
\lambda_1(\Omega_\eps)-\lambda_1(B)\geq\eps^3(C_1\lambda_1(B)-C\eps(1+C_2\eps^3)).
$$
If $\eps$ is sufficiently small, the right hand side is strictly positive. Moreover, 
\[
d_H(\Omega_\eps,B)\leq3\eps.
\]
The case of perimeter constraint is similar, it amounts to observing that the ratio $\frac{|\partial\Omega_{\eta(\eps),\eps}|}{|\partial B|}$ satisfies the analogue of \eqref{comp} with different constants $C_1,C_2>0$.

\smallskip

Provided that \eqref{conv_1} holds true, this shows that the ball is not a local maximiser for $\lambda_1(\Omega)$ under either volume or perimeter constraint, and, more in general, that there are no local maximisers homeomorphic to a ball for $\lambda_1$ under both volume and perimeter constraints  (just take any $\Omega_1$ homeomorphic to a ball and $\eps$ small enough).

\smallskip Convergence of eigenvalues and eigenfunctions on dumbbell-like closed manifolds for the eigenvalues of the Hodge Laplacian on $p$-forms has been investigated in detail in \cite{colboisanne}. We briefly describe the result. Suppose that $M$ is a closed, $n$-dimensional Riemannian manifold composed of two fixed parts, that is, two closed manifolds $M_1$, $M_2$, from which small geodesic balls of radius $\eta$ are removed. Then, $M_1,M_2$ (with a ball of radius $\eta$ removed from each) are joined by the channel $S_{\eta}=\mathbb S^{n-1}\times(0,s)$. In \cite{colboisanne} it is proved that, if $n=3$ and $p=1$, the limit spectrum is given by the union of the spectra of $M_1,M_2$ and of the Dirichlet and Neumann spectra of the interval $(0,s)$. In a forthcoming paper we aim at addressing the same problem for dumbbell domains and more in general, dumbbell-like manifolds with boundary and relative or absolute conditions and establish rigorously \eqref{conv_1}. In particular, for $n=3$ $p=1$ and relative conditions (the Maxwell's case) there is no contribution from the thin channel (Maxwell's eigenvalues on a thin cylinder go to $+\infty$), while we have the Dirichlet and Neumann spectrum at the limit in the case of absolute boundary conditions. We refer e.g., to \cite{guerini_savo} for an introduction on the Laplacian acting on $p$-forms on manifolds with boundary.

\begin{figure}
    \centering
\includegraphics[width=0.7\textwidth]{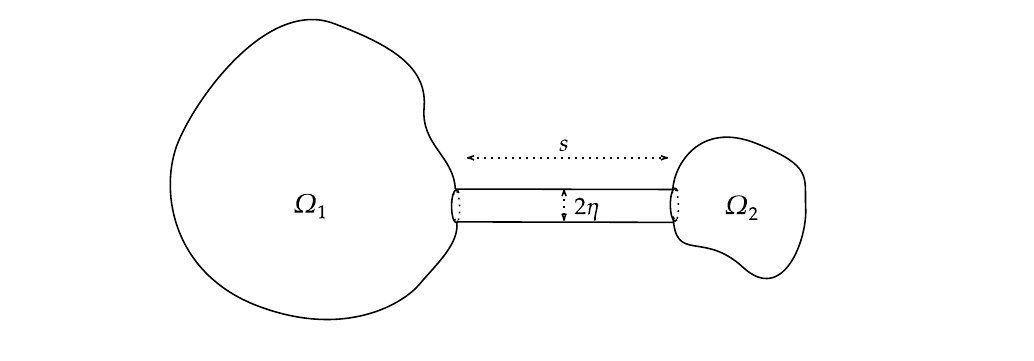}
    \caption{The dumbbell with one handle}
    \label{fig2}
\end{figure}

\section*{Acknowledgments}
The authors are thankful to Professor Bruno Colbois for fruitful discussions on eigenvalues on differential forms and for pointing out reference \cite{guerini}. The second author is grateful
to the Dipartimento di Matematica and the Dipartimento di Tecnica e Gestione dei Sistemi Industriali of the University of Padova for
the kind hospitality during the preparation of the manuscript. P.D. Lamberti and R. Sempio ackowledge the support of the SID project ``Spectral and Geometric Analysis of Electromagnetic Operators" funded by the Dipartimento di Tecnica e Gestione dei Sistemi Industriali, Università degli Studi di Padova.

\bibliography{bibliography.bib}
\bibliographystyle{abbrv}
\end{document}